%% file: channel-stability.tex
\documentclass[11pt,a4paper,twoside]{amsart}

\usepackage{graphicx}
\usepackage{amsfonts}
\usepackage{amssymb}
\usepackage{amsthm}
\usepackage[centertags]{amsmath}
\usepackage{newlfont}
\usepackage{hyperref}
\usepackage{color}
\usepackage{graphicx,color}
\usepackage{amsmath, amssymb, graphics}
\usepackage{xcolor}
 
	\usepackage{graphicx}
	\usepackage{xcolor}
	\usepackage{svg}
\def\r{\mathbb{R}}
 \def\c{\mathbf{c}}
\def\q{\mathbf{q}}
\def\n{\mathbb{N}}
\def\R{\mathbf{r}}
\def\e{\mathcal{E}}
 \def\s{\mathcal{S}} 
 
\newtheorem{theorem}{Theorem}[section]

\newtheorem{corollary}[theorem]{Corollary}
\newtheorem{lemma}[theorem]{Lemma}

\theoremstyle{definition}

\newtheorem{example}[theorem]{Example}
\newtheorem{remark}[theorem]{Remark}

\title[Capillary liquid channels   in  cylindrical support surfaces]{Capillary liquid channels   in  cylindrical support surfaces: stability and bifurcation} 
\author{Rafael L\'opez}  
\address{Departamento de Geometr\'{\i}a y Topolog\'{\i}a\\  Universidad de Granada. 18071 Granada, Spain} 
\email{rcamino@ugr.es}

\keywords{capillarity, cylindrical surface, stability, Plateau-Rayleigh instability,  bifurcation.}

\subjclass{76B45, 53A10, 34K18, 35J60, 58J55}

\begin{document} 
 
\begin{abstract}  
Planes and circular cylinders are models of interfaces of a fluid when the support surface is translationally invariant in a direction of the space. After a study of the eigenvalues of the Jacobi operator, it is investigated when planar strips and sections of circular cylinders are stable in   cylindrical symmetric support surfaces.  This analysis depends on the curvature of the support at the contact points with the interface. The Plateau-Rayleigh instability phenomenon is studied finding the critical value $h_0>0$ such that rectangular pieces of planar strips or circular cylinders of length greater than $h_0$ are necessarily unstable.    It is also studied when new morphologies of capillary surfaces can emerge from given circular cylinders. Using the method of bifurcation by simple eigenvalues, we establish conditions on the support surface that prove that  when $0$ is a simple eigenvalue of the Jacobi operator,  there is bifurcation from  explicit circular cylinders. It will be presented examples of supports (parabolic and catenary cylinders) where this bifurcation appears. 
\end{abstract} 
\maketitle

  %%%%%%%
\section{Introduction to the problem}
%%%%%%%%%%%%%%%%%%%%%%%%%%

Consider a stream of liquid $\Omega$ deposited  along a solid support  $\s$ under ideal conditions, such as homogeneity  of materials, rigidity, smoothness and zero hysteresis. It is  also assumed that the effect of gravity is negligible.  
When $\Omega$ attains its equilibrium in a static position, its shape   is given by the surface tension of the air-liquid interface $\Sigma$ and governed by the Laplace-Young equations \cite{la,yo}. The   interface $\Sigma$ minimizes locally the surface energy, which it is  given by   
\begin{equation}\label{e1}
\e(\Sigma)=\lvert\Sigma\rvert-\frac{\sigma_{SA}-\sigma_{SL}}{\sigma_{LA}} \lvert\Sigma^*\rvert,
\end{equation}
where $\lvert\Sigma\rvert$ is the area of $\Sigma$  and $\lvert\Sigma^*\rvert$ is the area of the region $\Sigma^*$ wetted by $\Omega$ in contact with $\s$.  Here $\sigma_{ij}$  denotes the interfacial surface tensions  of the solid (S), liquid (L) and air (A) phases. The inequality $\lvert\sigma_{SA}-\sigma_{SL}\rvert<\sigma_{LA}$ is assumed to assure the existence of wetting in mechanical equilibrium.   A critical point  of $\e$ is a   surface of constant mean curvature  $H$ (cmc surface in short). The constancy of the mean curvature of $\Sigma$ is the Laplace equation $ P_L-P_A=2H \sigma_{LA}$, where $P_L$ and $P_A$ are the pressures of the  liquid and air phases. A second condition is the known Young equation which  asserts that the contact angle $\gamma\in (0,\pi)$ between $\Sigma$ and the walls of $\s$ is constant, with $\cos\gamma=(\sigma_{SA}-\sigma_{SL})/\sigma_{LA}$.   A surface with constant mean curvature and with constant contact angle with the support is called   a capillary surface on $\s$.

Knowledge of the surface tensions is of special relevance in fluid theory because it can explain a variety of experimental phenomena \cite{dege}.  Many of the measurement techniques involve the determination of the shape of $\Sigma$ \cite{ad}, and thus the geometry of the different morphologies that can adopt a liquid. Although the hypothesis of zero gravity implies contexts of micro-gravity, it is also applicable when the weight of the liquid is almost negligible, as for example, in the interfaces of biological microstructures and self-assembly systems \cite{and,hyd,hys,tho}.

 Stable configurations of interfaces are interesting because they are physically realizable  \cite{bst1}.  A capillary surface is stable if for all compact perturbations of the surface, the energy $\e$ does not decrease. A reasonable assumption is that the volume of the fluid is preserved in these perturbations.  Historically, instability was firstly observed  by Plateau in 1871 when he realized that a stream of water dropping vertically breaks    in a set of spherical droplets with the same amount of volume but less surface area \cite{pl}.  It was Rayleigh who theoretically proved   that  a stream of radius $r$ breaks into drops if its length is greater than $2\pi r$ \cite{ra}.   In the theory of cmc surfaces,   two results on stability deserve to point out. Assuming that the surface is closed (compact without boundary), spheres are the  only stable closed cmc surfaces \cite{bc}. In case that the surface is spanned by a circle, then circular disks and spherical caps are the only stable compact cmc surfaces with the topology of a disk \cite{alp}.

The number of explicit examples of cmc surfaces that can be obtained analytically   is few  due that the mean curvature equation is non-linear \cite{lo0}.  A   way to obtain cmc surfaces   is imposing certain symmetries on the surface. In this paper we will assume that the support $\s$ is translationally  invariant in a spatial direction. Then the walls of $\s$ is a cylindrical surface   generated by a planar curve $\c$. We will denote $\s$ by  $\s(\c)$.  A cylindrical surface is formed by all straight-lines, called rulings, that are parallel   and   pass through the given planar curve $\c$. The rulings will be assumed to be horizontal lines. Since   the support is invariant in a spatial direction, it is plausible to admit that the liquid  deposited on $\s(\c)$  acquires the same symmetry. Therefore, the interface $\Sigma$ is also a cylindrical surface and we say that $\Sigma$ is a {\it liquid channel}. See Fig. \ref{fig1}. The classification of   cylindrical   cmc surfaces is well known, being  planes (if $P_L-P_A=0$) and circular cylinders (if $P_L-P_A\not=0$) \cite{lo0}.

\begin{figure}[hbtp]
\begin{center}\scalebox{1.2}{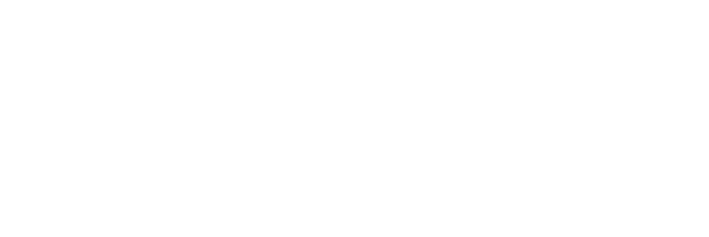 }
\end{center}
\caption{Cross section of a liquid channel $\Omega$ deposited on a cylindrical support $\s(\c)$. It is assumed that $\s(\c)$ is  symmetric about a vertical plane. The contact angle is determined by the unit normal vectors $N$ and $\widetilde{N}$ of $\Sigma$ and $\s(\c)$, respectively.  }\label{fig1}
\end{figure}

In the literature, the  supports usually studied have  been planes and cylinders. For planes, it is included the case of parallel planes and    wedges (two planes which intersect along a straight-line).  For these supports, circular cylinders are capillary surfaces provide its axis is parallel to the plane   and suitably placed in order to fulfill the   condition on the contact angle  \cite{concus,dav,lan,lo1,lo00,maj,vo,vo2,vo5,yh}. In case that  the liquid is confined in  strips of the plane, if the contact angle is lower than $\pi/2$, then the liquid channel is stable; if the angle is more than $\pi/2$, then $\Sigma$ is unstable with a number of unstable wavenumbers  \cite{benilov,bs,dav,herrada,rss}. An interesting case  is when the support is formed by planar strips that alternate hydrophilic and hydrophobic materials \cite{bkk,ghll,sl,ubal}. For the stability  of cmc surfaces connecting two parallel planes, see \cite{vo5,vo6}.
 
 Another support surface of cylindrical type  is a circular cylinder ($\c$ is a circle). Examples of capillary surfaces on $\s(\c)$ are planes parallel to the rulings of $\s(\c)$ and pieces of infinite circular cylinders whose axis is parallel to the rulings of $\s(\c)$. Vogel recently studied the stability of both surfaces proving that  if the length of the capillary surface is sufficiently large, then both surfaces  are unstable \cite{vo4} (see also \cite{yd}).

The goal of this paper is to investigate liquid channels for  a general support surface $\s(\c)$ of cylindrical type. As we said, we can think $\c=\c(s)$ as a planar curve contained in a vertical plane. Then $\s(\c)$ is the product $\c(s)+t\vec{v}$, $s\in I\subset\r$, $t\in\r$, where $\vec{v}$ is a unitary vector orthogonal to the plane containing $\c$. The support $\s(\c)$ can be parametrized by $(s,t)\mapsto \c(s)+t\vec{v}$. This work is motivated by \cite{rss}, where it was assumed that the liquid is pinned between two fixed straight-lines of the interface $\Sigma$. However, in the present paper, it  will be assumed that the boundary curve of $\Sigma$ can freely move  on the support surface.  Also the support surface $\s(\c)$ will be symmetric with respect to  a vertical plane $P$ and orthogonal to the plane containing the curve $\c$.  Under this symmetry condition, two known cmc surfaces $\Sigma$ are candidates as capillary surfaces on $\s(\c)$: an infinite    horizontal  planar strip and a  piece of an infinite circular cylinder whose axis is parallel to the rulings of $\s(\c)$. Stability of both surfaces are investigated in Sects. \ref{sec3} and \ref{sec4}.  Without to enter in the precise detail of the results, for horizontal planar strips  we prove:

\begin{quote} {\it Consider $\Sigma$ an infinite horizontal planar   strip which it is a capillary surface on $\s(\c)$. If the curvature $\kappa$ of $\c=\c(s)$ is positive (resp. non-positive) at the contact points $s=\tau_0$, $t\in\r$, then $\Sigma$ is unstable (resp. strongly stable). Moreover, if $\kappa(\tau_0)>0$, we give estimates of the length $h$ of $\Sigma$ (along the rulings) that ensures that if $h>h_0$ for some value $h_0>0$, then $\Sigma$ is unstable. }
\end{quote} 
The last statement extends the known Plateau-Rayleigh instability criterion to the case  of arbitrary shape of $\s(\c)$. A similar result is obtained when $\Sigma$ is an infinite piece of a circular cylinder of radius $r>0$ supported horizontally on $\s(\c)$. Here $\Sigma\cap\s(\c)$ are two horizontal  straight lines parallel to $\vec{v}$.
 
\begin{quote}{\it  Consider $\Sigma$ a circular cylinder of radius $r$ with axis parallel to the rulings of $\s(\c)$ and supported on   $\s(\c)$. Let $s=\tau_0$, $t\in\r$,  be the points of $\c$ that determine the contact between $\Sigma$ and $\s(\c)$. Let $\gamma$ be the contact angle. If $r\kappa(\tau_0)\pm \cos\gamma\geq 0$, then $\Sigma$ is unstable, where the sign $+$ or $-$ has to be precise depending on the configuration of the liquid.   If $r\kappa(\tau_0)\pm \cos\gamma<0$, we give criteria to distinguish if $\Sigma$ is or not stable.}
\end{quote}

 A further step in the study of the stability of  liquid channels is to ask whether new examples of capillary surfaces can emerge from the known ones. This question is natural thinking in the classical Plateau-Rayleigh experiment, where large cylindrical streams break in spherical droplets. In the present case, we ask on the existence of new  capillary surfaces bifurcating from   circular cylinders supported on $\s(\c)$. This is no new and phenomena of bifurcation  when the support is relatively simple have been studied in the literature:   sessile and  pendent droplets with circular boundary \cite{bru,we-c,we-c2}; liquid bridges between parallel planes \cite{kpm1,kpm2,vo3,zz}; and circular cylinders confined in strips \cite{bkk,ghll,lo2}  and wedges \cite{we-c2,lo1}.

 In this paper, the bifurcation appears after the study of instability of circular cylinders. When the support has vertical symmetry,  and fixing a contact angle $\gamma$, we will see that there is a   family of circular cylinders making the same contact angle $\gamma$ with the support and where   the radii of these cylinders vary in a certain interval of $\r$. The parameter used  in our results of bifurcation is the radius of these cylinders, or equivalently, the mean curvatures of the cylinders. For the precise statement of the bifurcation result, we refer to Sect. \ref{sec5}. Roughly speaking, we prove:
 \begin{quote}{\it  Consider $\Sigma$ a circular cylinder which it is supported horizontally on $\s(\c)$ such that its axis is parallel to the rulings of $\s(\c)$ and making the same contact angle along the two contact lines. Under certain conditions on $r\kappa(\tau_0)\pm\cos\gamma$ at the contact point $s=\tau_0$ between $\Sigma$ and $\s(\c)$, there is a one-parameter of cmc surfaces making a contact angle $\gamma$ with $\s(\c)$. These surfaces are periodic along the direction of the rulings. }
 \end{quote}
    By the Laplace equation, the mean curvature as a parameter  can be seen as a type of control of the pressures in both sides of the interface. 
    
    This paper is organized as follows.  In Sect. \ref{sec2} we present the mathematical framework, fixing the notation and obtaining the stability operator associated to the energy $\e$. In Sects. \ref{sec3} and \ref{sec4} we study the stability of planes and circular cylinders establishing a  Plateau-Rayleigh phenomenon of instability. We give an analysis of the eigenvalues of the Jacobi operator of $\e$. In Sect. \ref{sec5}, we ask for the existence of new examples of capillary surfaces on $\s(\c)$ other planes and circular cylinders.   First, we present the machinery of bifurcation necessary in our results which it is  based on the known case of simple eigenvalues of Crandall and Rabinowitz. It will be  established conditions that ensure that when $0$ is a simple eigenvalue of the Jacobi operator,   new examples of capillary surfaces bifurcate from a given circular cylinder. To illustrate how the method works, it will be discussed  in Sect. \ref{sec6}  examples of support surfaces such as parabolic and catenary cylinders when the contact angle is $\gamma=\pi/2$.

%%%%%%%%%%%%%%%%%%%%%%%%%%%%%%%%%%%%%%%%%%%%%
\section{The stability operator}\label{sec2}
%%%%%%%%%%%%%%%%%%%%%%%%%%%%%%%%%%%%%%%%%%%

In this section,    it will be fixed the notation employed along this paper and we will recall the stability operator.  Let $\s$ be a  surface without self-intersections  (the support surface) that separates $\r^3$ in two domains and let us denote by $W$ one of them.   Let $\Sigma$ be an orientable    surface with possible non-empty boundary $\partial\Sigma$. We say that an immersion   $X\colon\Sigma\to\r^3$ is {\it admissible} in $W$ if  $X(\Sigma)$ is contained $W$ and $X(\partial\Sigma)$  is contained in $\s$.  If there is not confusion, we will write $\Sigma$ (resp. $\partial\Sigma$)  instead of $X(\Sigma)$ (resp. $X(\partial\Sigma)$). It is assumed the physical assumption that the boundary $\partial\Sigma$ of $\Sigma$ moves freely on $\s$ whereas the interior of $\Sigma$ is contained in $W$.  The surface $\Sigma$ together some regions $\Sigma^*$ of $\s$ enclose a domain $\Omega$ of $\r^3$  (the fluid).  By $\Sigma^*$ we denote the region of $\s$ wetted by the fluid,   $\Sigma^*=\partial\Omega\cap \s$.

An {\it admissible variation} of $X$ is a smooth map $\overline{X}:\Sigma\times (-\epsilon,\epsilon)\to \r^3$ such that $\overline{X}(p,0)=X(p)$ and the maps $X_t\colon p\mapsto \overline{X}(p,t)$ are   admissible immersions of $\Sigma$ in $W$ for all $|t|<\epsilon$. For a fixed $t$, define $\Sigma_t=X_t( \Sigma)$ and $\Omega(t)$ the domain bounded by $\partial\Sigma_t$ in $\partial W$. We will assume that all variations have compact support, that is, for each $t$ there is a compact set $K\subset\Sigma$ such that $X_t(p)=X(p)$ for all $p\in\Sigma\setminus K$.

Let $\gamma\in (0,\pi)$. The surface $\Sigma$ is called a {\it capillary surface} if it is a critical point of the energy   \eqref{e1} for all compactly supported volume-preserving admissible variations. A  capillary surface  is characterized by the constancy of its mean curvature $H$ and by the fact that the angle between $\Sigma$ and $\s$ along $\partial\Sigma$ is constant and it coincides with $\gamma$. The angle $\gamma$ is the angle formed by the unit normal vectors $N$ of $\Sigma$ and $\widetilde{N}$ of $\s$ along $\partial\Sigma$, $\cos\gamma=\langle N,\widetilde{N}\rangle$. The vector   $N$ points into the fluid $\Omega$ whereas $\widetilde{N}$ points outwards $\Omega$ (Fig. \ref{fig1}).

From now on, all support surfaces  will be cylindrical surfaces and symmetric about a vertical plane $P$. In order to make precise, let us introduce   the following notation. Let $(x,y,z)$ be canonical coordinates of Euclidean space $\r^3$ and let $\langle,\rangle$ be the Euclidean metric of $\r^3$. The terms vertical and horizontal will indicate  to be parallel to the $z$-axis or to the $xy$-plane respectively. Consider a planar curve $\c=\c(\tau)$, $\tau\in I$, where $I\subset\r$ is an interval, and suppose that $\c$ is contained in the  $xz$-plane,  $\c(\tau)=(x(\tau),0,z(\tau))$. Let $\s(\c)$ be the cylindrical surface determined by $\c$ whose rulings  are parallel to the $y$-line, $\s(\c)=\{\c(\tau)+t(0,1,0): \tau\in I,t\in\r\}$. A parametrization of $\s(\c)$ is 
\begin{equation}\label{xx}
X(\tau,t)=(x(\tau),t,z(\tau)),\quad \tau\in I,t\in\r.
\end{equation}
About the vertical symmetry, without loss of generality, we  will assume that the vertical plane of symmetry is   the $yz$-plane. If the domain $I$ is symmetric about $0\in I$, the vertical symmetry means that $\c(-\tau)=(-x(\tau),0,z(\tau))$, $\tau\in I$. 
  
  We show that planes and circular cylinders are examples of capillary surfaces on symmetric supports $\s(\c)$.

\begin{example}[Planar strips]\label{ex1} {\rm 
Let $\Pi$ be the horizontal plane of equation $z=c$ and suppose that $\Pi$ intersects $\s(\c)$. Recall that $H=0$. On the other hand, $\Pi$ is also a cylindrical surface whose generating curve is the straight-line $\R(s)=(s,0,c)$, $s\in\r$. The intersection of $\R$ with $\c$ can occur in many (symmetric) points (Fig. \ref{fig22}, left). However, the Young condition requires that the angle is the same in all these points. Therefore, we will simplify the situation by assuming that $\R$ and $\c$ only intersect at two (symmetric) points $s=\pm s_0$ (Fig. \ref{fig22}, middle). Then the capillary surface is the strip of $\Pi$  determined by both points. Note that the part of the plane outside $\s(\c)$ is also a capillary surface (Fig. \ref{fig22}, right). This surface has its own interest, but the arguments in the study of the stability cannot be applied to this situation.}
\end{example}

  \begin{figure}[hbtp]
\begin{center}\scalebox{1}{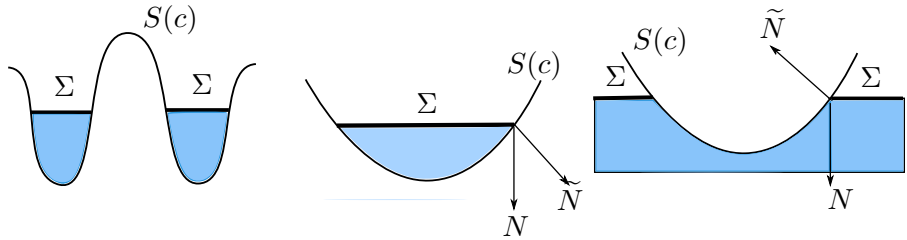 }
\end{center}
\caption{ A plane on $\s(\c)$ with different contact angles (left). A capillary strip on $\s(\c)$ (middle). A capillary plane outside of $\s(\c)$ (right). }\label{fig22}
\end{figure}

\begin{example}[Sections of circular cylinders]\label{ex2}
{\rm Let $C_r$ be a circular cylinder of radius $r>0$ whose axis is contained in the $yz$-plane and parallel to the $y$-axis. Then $C_r$ has constant mean curvature $H=1/(2r)$ (with the inward orientation). The surface $C_r$ is a cylindrical surface generated by a circle of radius $r>0$ centered at the $z$-axis. This circle can be parametrized by $\R(s)=  (r\sin(s),0,m+r\cos(s))$, $s\in\r$. As in Ex. \ref{ex1}, the intersection between $\R$ and  $\c$ can occur in many points but we will assume that this intersection only happens at two points, namely, $s=-s_0$ and $s= s_0$. Both points determine two arcs in $\R$, namely, $\{ -s_0\leq s\leq s_0\}$ and $\{s_0\leq s\leq 2\pi-s_0\}$. Each one these arcs define a subset in  $C_r$ which is called a section of a circular cylinder. Both sections circular cylinders are capillary surfaces on  $\s(\c)$.  

}\end{example}

We now recall  the formula of the second variation of $\e$. Let     $C^\infty_0(\Sigma)$ be the space of compactly supported smooth functions. A function $u\in C^\infty_0(\Sigma)$ determines a variation $\Sigma_t$ where $u$ is     the normal component of the variational vector field. The condition that the first variation of the enclosed volume by   $\Sigma_t$ vanishes  is equivalent to $\int_\Sigma u\, d\Sigma=0$. The expression of $\e''(0)$ is
\begin{equation}\label{eq1}
\e''(0)=-\int_\Sigma u(\Delta u+\lvert A\rvert^2 u)\, d\Sigma+\int_{\partial \Sigma}
u\Big(\frac{\partial u}{\partial\nu}-\q u\Big) ds,
\end{equation}
where
\begin{equation}\label{qq}
\q=\csc{\gamma} \widetilde{A}(\widetilde{\nu},\widetilde{\nu})+\cot{\gamma}\, A(\nu,\nu).
\end{equation}
 See \cite{rv,rss}. Here $\Delta$ is the Laplacian operator on $\Sigma$, $A$ is the second fundamental form of $\Sigma$ with respect to $N$ and $\widetilde{A}$ is the second fundamental  of $\s$ with respect to $-\widetilde{N}$.  The vectors $\nu$ and $\widetilde{\nu}$ are the exterior  unit conormal vectors of $\partial\Sigma$ on $\Sigma$ and on $\s$ respectively. If we are assuming that $\s$  is a cylindrical surface $\s(\c)$, then $\widetilde{A}(\widetilde{\nu},\widetilde{\nu})$ coincides with the curvature $\kappa$ of the curve $\c$ at the intersection point.

A capillary surface  is said to be {\it stable} if   $\e''(0)\geq 0$  for all compactly supported volume-preserving admissible variations.    From the expression of the second variation $\e''(0)$ in \eqref{eq1},  a quadratic form $Q$  on  the space ${\mathcal V}=\{u\in C^\infty(\Sigma):\int_\Sigma u\, d\Sigma=0\}$ is  defined by
\begin{equation}\label{equa}
Q[u]=\int_\Sigma\left(\lvert\nabla u\rvert^2-\lvert A\rvert^2u^2\right) \, d\Sigma-\int_{\partial \Sigma}\q u^2\, ds,
\end{equation}
where $\nabla$ denotes   the gradient on $\Sigma$. Thus,  $\Sigma$ is stable  if and only if $Q[u]\geq 0$ for all $u\in{\mathcal V}$. The quadratic form $Q$ defines the so-called {\it Jacobi operator} $\mathcal{L}$ given by
\begin{equation}\label{l1}
\mathcal{L}=\Delta+\lvert A\rvert^2.
\end{equation}
Recall that $|A|^2=4H^2-2K$,  where $K$ the Gauss curvature of $\Sigma$.  

If $\Sigma$ is not stable, it is interesting to study the subspace of ${\mathcal V}$ where $Q$ is negative definite. The  {\it weak Morse index} of $\Sigma$, denoted by $\mbox{index}_w(\Sigma)$, is defined as the maximum dimension of any subspace of ${\mathcal V}$ on which $Q$ is negative definite.    If $\Sigma$ is compact, the weak Morse index of $\Sigma$ coincides with the number of negative eigenvalues $\lambda_w$ of the   eigenvalue problem     
\begin{equation}\label{eq10}
\left\{\begin{split}
\mathcal{L}u+\lambda_w u=0&\mbox{ in}\ \Sigma,\\
\frac{\partial u}{\partial \nu} - \q u=0& \mbox{ in}\  \partial\Sigma,\\
u&\in{\mathcal V}
\end{split}\right.
\end{equation}
Since the condition $u\in  {\mathcal V}$ is difficult to work with,  instead of \eqref{eq10}, we consider the eigenvalue problem
\begin{equation}\label{p-e}
\left\{\begin{split}
\mathcal{L}u+\lambda u=0&\mbox{ in}\ \Sigma,\\
\frac{\partial u}{\partial \nu} - \q u=0& \mbox{ in}\  \partial\Sigma,\\
u& \in C^\infty(\Sigma)
\end{split}\right.
\end{equation}
By the ellipticity of $\mathcal{L}$, it is well-know that the eigenvalues of \eqref{p-e} (also $\lambda_w$ of \eqref{eq10}) are ordered as a discrete spectrum  $\lambda_1<\lambda_2 \leq\lambda_3\cdots\nearrow \infty$ counting multiplicity.   The {\it Morse index} of $\Sigma$, denoted $\mbox{index}(\Sigma)$, is the number of negative eigenvalues of \eqref{p-e}. If $\mbox{index}(\Sigma)=0$, we say that the surface is {\it strongly stable} and this is equivalent to $Q[u]\geq 0$ for all $u\in C^\infty(\Sigma)$ regardless the condition that $\int_\Sigma u=0$. Both indices are related by the inequalities   
\begin{equation}\label{ii}
\mbox{index}_w(\Sigma)\leq \mbox{index}(\Sigma)\leq \mbox{index}_w(\Sigma)+1.
\end{equation}
For example, if $\lambda_1\geq 0$, then    the surface is stable because  $\mbox{index}(\Sigma)=0$. This means that strong stability implies stability. If   $\lambda_2<0$, then  the subspace spanned by  two eigenfunctions of $\lambda_1$ and $\lambda_2$ has at least dimension $2$. This allows to construct   a function $u\in\mathcal{V}$ with $Q[u]<0$, proving that     $\Sigma$ is not stable.  For   a relation between the weak Morse index and the Morse index, we refer the reader to \cite{ko,vo2,vo12}.

 In case that $\Sigma$ is not compact, then we need to take an exhaustion of the surface.  If $\Sigma_1\subset\Sigma_2\subset\ldots\subset  \Sigma$ is an exhaustion of $\Sigma$ by bounded subdomains, the weak Morse index and the Morse index of $\Sigma$ are defined by 
 \begin{equation}\label{stable-infinito}
 \mbox{index\,}_w(\Sigma)=\lim_{n\to\infty}\mbox{index\,}_w(\Sigma_n),\quad \mbox{index\,}(\Sigma)=\lim_{n\to\infty}\mbox{index\,}(\Sigma_n).
 \end{equation}
These definitions  are independent of the choice of the exhaustion of $\Sigma$. Both numbers can be infinite, but if they are finite, then the relation \eqref{ii} holds too. We point out that   for two bounded open domains $\Sigma'$, $\Sigma''$, if $\Sigma'\subset\Sigma''$, then $\mbox{index\,}_w(\Sigma')\geq \mbox{index\,}_w(\Sigma'')$ and $\mbox{index\,}(\Sigma')\geq \mbox{index\,}(\Sigma'')$.

%%%%%%%%%%%%%%%%%%%%%%%%%%
\section{Stability of  horizontal planes}\label{sec3}
%%%%%%%%%%%%%%%%%%%%%%%%%%%%%%%%%%

 In this section we study the stability of planar strips as  capillary surfaces on $\s(\c)$. Stability depends on the curvature of the support surface at the intersection points with the plane. If the curvature $\kappa$ is positive, then planar strips  will be unstable and we will obtain the critical length of the strip that determines the  Plateau-Rayleigh  instability phenomenon. If $\kappa$ is non-positive, then the  planar strip will be stable. See Thm.  \ref{t1} below.   
 
  Let  $\Pi$ be a horizontal plane of equation $z=c$ and suppose that $\Pi$ intersects $\s(\c)$.   By the symmetry of $\s$,  this intersection  is determined by the intersection of $\c$ with the horizontal line $\R(s)=(s,0,c)$. In order to distinguish the parameter of $\c$, we will denote $\c=\c(\tau)$, $\tau\in I$, where we also assume that the domain $I$ of $\c$ is symmetric about $0\in I$. Suppose that this intersection occurs only at  the points $s=\pm s_0$, for some $s_0>0$, with $\R(s_0)=\c(\tau_0)$ for some $\tau_0$. Thus $\Pi$ and $\s(\c)$ determine a strip  in $\Pi$. The boundary of this strip is formed by the two parallel horizontal straight-lines through $\R(\pm s_0)$. These lines are   $L_{s_0}^+=\{x= s_0, z=c\}$ and $L_{s_0}^{-}=\{x=-s_0,z=c\}$.  This strip is the capillary surface $\Sigma$ to investigate.

\begin{figure}[hbtp]
\begin{center}\scalebox{1}{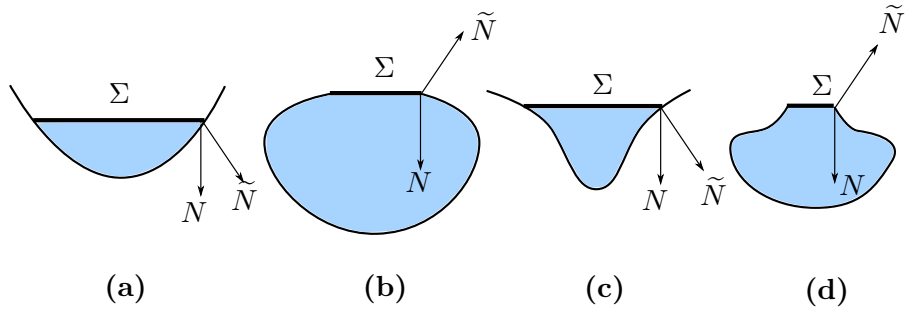 }
\end{center}
\caption{A capillary planar strip on $\s(\c)$. The different types are: $\widetilde{A}>0$ and $\gamma\in (0,\frac{\pi}{2})$ (a);  $\widetilde{A}>0$ and $\gamma\in ( \frac{\pi}{2},\pi)$ (b); $\widetilde{A}<0$ and $\gamma\in (0,\frac{\pi}{2})$ (c);  $\widetilde{A}<0$ and $\gamma\in (\frac{\pi}{2},\pi)$ (d).  }\label{fig3}
\end{figure}

Without loss of generality, we can assume that the fluid $\Omega$ lies below $\Pi$: if $\Omega$ lies above $\Pi$, the arguments are analogous but the contact angle is supplementary.  The orientation on $\Pi$ is $N=(0,0,-1)$.  Since $\s(\c)$ is a cylindrical surface, then $\q=\widetilde{A}(\widetilde{\nu},\widetilde{\nu})$ is the curvature $\kappa$ of $\c$ with the corresponding sign because $\widetilde{A}$ is computed according to the orientation $-\widetilde{N}$.  Let   $\widetilde{A}(\widetilde{\nu},\widetilde{\nu}) =\kappa(\tau_0)$, where   $\kappa(\tau_0)$ denotes the curvature of $\c$ at the intersection point with $\c(\tau_0)$. On the other hand,  the strip $\Sigma$ is bounded by $-s_0\leq x\leq s_0$. Thus  $\nu(\tau_0)=(1,0,0)$ and $\nu(-\tau_0)=(-1,0,0)$.  

Since we need  to consider compact perturbations of the surface, we will take rectangular pieces of $\Sigma$ situated between the vertical planes of equations $y=0$ and $y=h$.  We will calculate the eigenvalues associated with the eigenvalue problem  \eqref{p-e}. Let us parametrize $\Sigma$  by  $x=s$, $y=t$ and $z=c$. Notice that  $\Delta$ coincides with the Euclidean Laplacian $\Delta u=u_{ss}+u_{tt}$. We also have 
$$\frac{\partial u}{\partial\nu}(s_0)=u_s(s_0),\quad \frac{\partial u}{\partial\nu}(-s_0)=-u_s(s_0).$$
Because $A=0$, we have  $\q=  \kappa/\sin\gamma$. 
 Then the eigenvalue problem  \eqref{p-e}   modified by adding the boundary condition on $y=0$ and $y=h$ is
 \begin{equation}\label{ei1}
\left.
\begin{aligned}
u_{ss}+u_{tt}+\lambda u&=0,\\
u_s(s_0,t)-\frac{\kappa(\tau_0)}{\sin\gamma} u(s_0,t)&=0,\\
u_s(-s_0,t)+\frac{\kappa(\tau_0)}{\sin\gamma}u(-s_0,t)&=0,\\
u(s,0)=u(s,h)&=0.
\end{aligned}
\right\}
\end{equation}
We are using that $\kappa(\tau_0)=\kappa(-\tau_0)$ by the symmetry of $\c$. 

\begin{remark}
{\rm If $\c$ is  the graph of a function $\phi=\phi(x)$, the expression of $\kappa$ is 
\begin{equation}\label{k}
\kappa=\pm\frac{\phi''}{(1+\phi'^2)^{3/2}},
\end{equation}
 where the sign of $\kappa$ coincides with that of  $\widetilde{A}$. We also have     $N=(0,0,-1)$ and  $\widetilde{N}=\pm\frac{1}{\sqrt{1+\phi^2}}(-\phi',0,1)$. Thus $\sin\gamma= \lvert\phi'\rvert/\sqrt{1+\phi'^2} $ and 
\begin{equation}\label{kk}
\frac{\kappa}{\sin\gamma}=\frac{\phi''}{\lvert\phi'(s)\rvert(1+\phi'^2)}.
\end{equation}}
\end{remark}

We solve \eqref{ei1}   by the method of separation of variables. The following computations can be considered as elementary calculations. However, in order to give rigor to the arguments, we have included them in the article despite the fact that they seem cumbersome and tedious.   By the last boundary condition in \eqref{ei1}, a general solution can be expressed as 
$$u(t,s)=\sum_{n=1}^\infty f_n(s)\sin\left(\frac{n\pi t}{h}\right),\quad s\in [-s_0,s_0], t\in [0,h],$$
where $f_n$ are   functions defined in $[-s_0,s_0]$. In order to simplify the notation, we will drop  the subscript $n$ for $f_n$.  The first three equations in \eqref{ei1} are equivalent to
\begin{equation}\label{b1}
\left.
\begin{aligned}
f''(s)-\left(\frac{n^2\pi^2}{h^2}-\lambda\right)f(s)&=0,\\
f'(s_0)-\frac{\kappa(\tau_0)}{\sin\gamma} f(s_0)&=0,\\
f'(-s_0)+\frac{\kappa(\tau_0)}{\sin\gamma}f(-s_0)&=0.
\end{aligned}
\right\}
\end{equation}
  The solution of the first equation in \eqref{b1} depends on   the sign of $\frac{n^2\pi^2}{h^2}-\lambda$. 
  \begin{enumerate}
\item Case $\frac{n^2\pi^2}{h^2}-\lambda>0$.   Let $\beta>0$ be the number defined by
$$\beta^2=\frac{n^2\pi^2}{h^2}-\lambda.$$
The solution of the first equation of \eqref{b1} is
$$f(s)=Ae^{\beta s}+Be^{-\beta s},$$
 for some constants $A$ and $B$.  Note that the solution $f(s)$ can be also expressed  as a linear combination of the hyperbolic trigonometric functions $\sinh(\beta s)$ and $\cosh(\beta s)$.  Imposing the boundary conditions \eqref{b1}, we have 
\begin{equation*}
\begin{split}
Ae^{\beta s_0}\left(\beta-\frac{\kappa(\tau_0)}{\sin\gamma}\right)+Be^{-\beta s_0}\left(-\beta-\frac{\kappa(\tau_0)}{\sin\gamma}\right)&=0,\\
Ae^{-\beta s_0}\left(\beta+\frac{\kappa(\tau_0)}{\sin\gamma}\right)+Be^{\beta s_0}\left(-\beta+\frac{\kappa(\tau_0)}{\sin\gamma}\right)&=0,
\end{split}
\end{equation*}
There is a non-trivial solution of this system if and only if the determinant of the coefficients $A$ and $B$ vanishes, that is:
$$e^{4\beta s_0}\left(\beta-\frac{\kappa(\tau_0)}{\sin\gamma}\right)\left(-\beta+\frac{\kappa(\tau_0)}{\sin\gamma}\right)=
\left(-\beta-\frac{\kappa(\tau_0)}{\sin\gamma}\right)\left(\beta+\frac{\kappa(\tau_0)}{\sin\gamma}\right),$$
or equivalently,
\begin{equation}\label{4t}
e^{4\beta s_0}=\left(\frac{\beta\sin\gamma+\kappa(\tau_0)}{\beta\sin\gamma-\kappa(\tau_0)}\right)^2.
\end{equation}
Define the function 
$$p(\beta)= e^{4\beta s_0}-\left(\frac{\beta\sin\gamma+\kappa(\tau_0)}{\beta\sin\gamma-\kappa(\tau_0)}\right)^2.$$
\begin{enumerate}
\item Case $\kappa(\tau_0)>0$.  Let $T= \kappa(\tau_0)/\sin\gamma$. Since $T>0$,  the domain of $h$ is $(0,T)\cup (T,\infty)$.  A study of the function $p$ proves that  $p$ is monotonic in each subdomain, being $p$ decreasing in $(0,T)$ and  increasing in $(T,\infty)$. Since $p(0)=0$, then $p$ is negative in $(0,T)$. In the interval $(T,\infty)$, we have $\lim_{\beta\to T^{+}}p(\beta)=-\infty$ and $\lim_{\beta\to\infty}p(\beta)=\infty$. Thus there is a unique $\beta\in (T,\infty)$ such that $p(\beta)=0$. For this value $\beta$,   the eigenvalues are
$$\lambda_n=\frac{n^2\pi^2}{h^2}-\beta^2,\quad n\in\mathbb{N}.$$
For large values   $h$ of the separation between the vertical walls $y=0$ and $y=h$, we obtain many negative eigenvalues, proving that $\Sigma$ is unstable.  

 \item Case $\kappa(\tau_0)\leq0$.   The function $p(\beta)$ is defined in $(0,\infty)$. Depending on the values $\kappa(\tau_0)$, $\sin\gamma$ and $s_0$, the function $h$ is either monotonically increasing or with a unique minimum $\beta_{min}$. In the first case, since $p(0)>0$ we have  that $p$ is positive; in the second one, $p$ has a minimum   $\beta_{m}$ with $p(\beta_{m})>0$. In both cases  the function $p$ is positive and   consequently, there is no  solution of $p(\beta)=0$. This implies that there are no solutions of \eqref{b1}.   
\end{enumerate}
\item Case $\frac{n^2\pi^2}{h^2}-\lambda=0$.  
The solution of the first equation of \eqref{b1} is $f(s)=A +Bs$ for some constants $A$ and $B$. The boundary conditions in \eqref{b1} imply
\begin{equation*}
\begin{split}
\kappa(\tau_0) A-(\sin\gamma-s_0\kappa(\tau_0))B&=0,\\
\kappa(\tau_0) A+(\sin\gamma-s_0\kappa(\tau_0))B&=0.
\end{split}
\end{equation*}
There are non-trivial solutions $A$ and $B$  of this system  if and only if $\kappa(\tau_0)(\sin\gamma-s_0\kappa(\tau_0))=0$. 
\begin{enumerate}
\item Case $\kappa(\tau_0)=0$.  Then $\sin\gamma-s_0\kappa(\tau_0)\not=0$, $B=0$ and $f(s)=A$. 
\item Case $s_0\kappa(\tau_0)=\sin\gamma$. Then $\kappa(\tau_0)\not=0$, $A=0$ and $f(s)=Bs$.
\end{enumerate}
In both cases, the eigenvalues are $\lambda_n=\frac{n^2\pi^2}{h^2}$, $n\in\mathbb{N}$ and  all are positive.
\item Case $\frac{n^2\pi^2}{h^2}-\lambda<0$. Let $\beta>0$, where
$$\beta^2=-\frac{n^2\pi^2}{h^2}+\lambda.$$
The solution of the first equation of \eqref{b1} is
$$f(s)=A\cos(\beta s)+B\sin(\beta s),$$
 for some constants $A$ and $B$.    The boundary conditions \eqref{b1} yield 
 \begin{equation*}
\begin{split}
\left(\beta\sin(\beta s_0)+\frac{\kappa(\tau_0)}{\sin\gamma}\cos(\beta s_0)\right)A-\left(\beta\cos(\beta s_0)-\frac{\kappa(\tau_0)}{\sin\gamma}\sin(\beta s_0)\right)B&=0,\\
\left(\beta\sin(\beta s_0)+\frac{\kappa(\tau_0)}{\sin\gamma}\cos(\beta s_0)\right)A+\left(\beta\cos(\beta s_0)-\frac{\kappa(\tau_0)}{\sin\gamma}\sin(\beta s_0)\right)B&=0.
\end{split}
\end{equation*}
There are non-trivial solutions $A$ and $B$ if and only if 
\begin{equation}\label{c-tan}
\beta\sin(\beta s_0)+\frac{\kappa(\tau_0)}{\sin\gamma}\cos(\beta s_0)=0,\quad\mbox{ or }\quad\beta\cos(\beta s_0)-\frac{\kappa(\tau_0)}{\sin\gamma}\sin(\beta s_0),
\end{equation}
or equivalently, if 
$$\tan(\beta s_0)=-\frac{\kappa(\tau_0)}{\beta\sin\gamma},\quad\mbox{or}\quad \tan(\beta s_0)= \frac{\beta\sin\gamma}{\kappa(\tau_0)}.$$
  The function $\beta\mapsto\tan(\beta s_0)$ is increasing and defined in $(0,\infty)\setminus\{m_k:k\in\mathbb{N}\}$, where $m_k=\frac{(2k-1)\pi}{2 s_0}$, $k\in\mathbb{N}$. Since $\lim_{\beta\to (\frac{(2k-1)\pi}{2 s_0})^{\pm}}\tan(\beta s_0)=\mp\infty$ and the functions $\beta\mapsto 1/\beta$ and $\beta\mapsto \beta$  are monotonic, there is a unique solution $\beta_{k+1}$ of each one of the equations \eqref{c-tan} in each interval of type $I_k:=(m_k,m_{k+1})$. For each  $\beta_k$, we have 
$\lambda_{n,k}=\beta_k^2+\frac{n^2\pi^2}{h^2}>0$, $n\in\mathbb{N}$. Then all the eigenvalues $\lambda_{n,k}$ are positive.  
\end{enumerate}
 
 We summarize the above calculations. 

\begin{theorem} \label{t1}
Let $\s(\c)$ be a symmetric support and let $\Sigma$ be a planar strip  which it is  a capillary surface on $\s(\c)$ along $L_{s_0}^{+}\cup L_{s_0}^{-}$. 
\begin{enumerate}
\item If  $\kappa(\tau_0)>0$, then $\Sigma$ is unstable. In fact,  if the length $h$ of a rectangular piece of $\Sigma$ is greater than $h_0= 2\pi/\beta$, where $\beta$ is the solution of \eqref{4t}, then $\Sigma$ is unstable (Fig. \ref{fig3}, (a) and (b)). 
\item If $\kappa(\tau_0)\leq 0$, then $\Sigma$ is strongly stable (Fig. \ref{fig3}, (c) and (d)).
\end{enumerate}
\end{theorem}

\begin{proof} It only remains how to calculate $h_0$ of the item (1). The critical value $h_0$ appears when the first eigenvalue $\lambda_1$ is negative and the second one is just $0$. Doing $\lambda_2=0$, we obtain $h_0=2\pi/\beta$.  
\end{proof}

 The   statement (1) of Thm.  \ref{t1}  is the version of the Plateau-Rayleigh instability criterion adapted to our context. 
Let us also  observe that (2) of Thm.  \ref{t1}  can be directly deduced from the expression of $Q$ in \eqref{equa}, because $A=0$ and the value of $\q$ along $\partial\Sigma$ is  $\q=\csc\gamma\kappa(\tau_0)\leq 0$ if $\kappa(\tau_0)\leq 0$. This gives $Q[u]\geq 0$ for all $u\in\mathcal{V}$.

We show examples of Thm.  \ref{t1} when the support surfaces are circular cylinders and parabolic cylinders.

\begin{example}\label{ex3}{\rm    $\s(\c)$ is a circular cylinder (see also \cite{vo4}). We  assume that $\Sigma$ is the planar strip  situated at height $z=0$ given by $-x_0\leq x\leq x_0$. The support is the  circular cylinder of radius $r$ where $\c$ is the piece of a circle of radius $r$ that intersect the $x$-axis at $\pm x_0$ and whose center is $(0,0,\sqrt{r^2-x_0^2})$. Then $\sin\gamma=x_0/r$ and   \eqref{4t} is 
$$e^{4\beta x_0}=\left(\frac{ \beta x_0+1}{\beta x_0-1}\right)^2.$$
The solution of this equation is $\beta= 1.2/x_0$. Hence the value $h_0$ is 
$$h_0=\frac{2\pi}{1.2} x_0=5.23\, x_0.$$
}
\end{example}

\begin{example}\label{ex4} {\rm   $\s(\c)$ is a parabolic cylinder.   Let $\c$ be the parabola   $z(x)=x^2$. Suppose that $\Sigma$ is the strip of the plane at height $z=1$, with $-1\leq x\leq 1$.  Using \eqref{kk}, 
$$\frac{\kappa(x)}{\sin\gamma}=\frac{1}{x(1+4x^2)}.$$
Then $s_0=\tau_0=1$ and   the solution of \eqref{4t}  is $\beta= 0.462$ and thus $h_0= 13.58$. }
\end{example}

%%%%%%%%%%%%%%%%%%%%%%%%%%
\section{Stability of  circular cylinders}\label{sec4}
%%%%%%%%%%%%%%%%%%%%%%%%%%%%%%%%%%

Let  $\Sigma$ be a circular cylinder whose axis is parallel to the $y$-axis and contained in the $yz$-plane. Suppose that the circle $\R$ generating $\Sigma$ intersects  the curve $\c$ exactly at two points.  Both points determine two sections of circular cylinders which are capillary surfaces on $\s(\c)$. For convenience in the symmetry of the domain of the circle $\R$ that determines $\Sigma$, suppose that $\R$  is parametrized by $\R(s)= (r\sin(s),0,m+r\cos(s))$. Then the sections of the circular cylinder are given by the intervals   $[-s_0,s_0]$ and $[s_0,2\pi-s_0]$ in the $s$-domain of $\R$.  As a parametrization of $\Sigma$, we take 
\begin{equation}\label{pc}
X(t,s)=\R(s)+t(0,1,0),\quad s\in [-s_0,s_0], t\in\r.
\end{equation} 

Without loss of generality, we can assume that the liquid $\Omega$ lies below  the cylinder $\Sigma$  such as it is shown in  Figs. \ref{fig4} and \ref{fig5}. We will assume that the intersection between $\s(\c)$ and $\Sigma$ occurs at the parallel horizontal straight-lines  $L_{s_0}\cup L_{s_0}^{-}$ determined by the intersection points $\R(\pm s_0)=\c(\pm\tau_0)$. The normal curvature   $A(\nu,\nu)$ on $\Sigma$ in the direction of $\nu$ is the curvature of the circle $\R$. Then $A(\nu,\nu)$ is $-1/r$ or $1/r$ depending if $\Sigma$ is concave (Fig. \ref{fig4}) or convex (Fig. \ref{fig5}). Analogously,  the normal curvature $\widetilde{A}(\widetilde{\nu},\widetilde{\nu})$ of the cross-section of $\s(\c)$ is the curvature $\kappa$ of the curve  $\c$ and its sign depends on the orientation of $\s(\c)$ according to $-\widetilde{N}$.

 \begin{figure}[hbtp]
\begin{center}\scalebox{1}{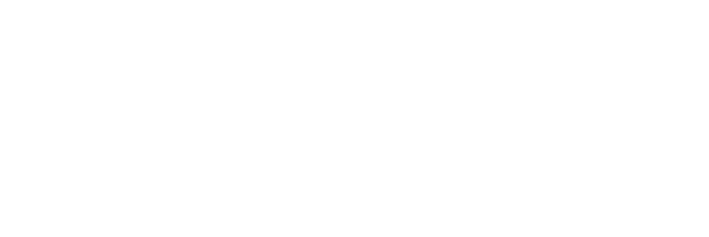 }
\end{center}
\caption{Case $A(\nu,\nu)=-1/r$. The different types are: $\widetilde{A}>0$ and $\gamma\in (0,\frac{\pi}{2})$ (a);  $\widetilde{A}>0$ and $\gamma\in ( \frac{\pi}{2},\pi)$ (b); $\widetilde{A}<0$ and $\gamma\in (0,\frac{\pi}{2})$ (c);  $\widetilde{A}<0$ and $\gamma\in (\frac{\pi}{2},\pi)$ (d).   }\label{fig4}
\end{figure}

 \begin{figure}[hbtp]
\begin{center}\scalebox{1}{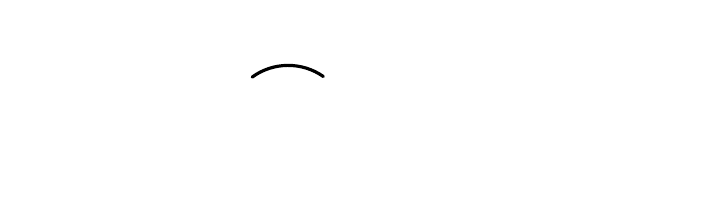 }
\end{center}
\caption{Case $A(\nu,\nu)=1/r$. The different types are: $\widetilde{A}>0$ and $\gamma\in (0,\frac{\pi}{2})$ (a);  $\widetilde{A}>0$ and $\gamma\in ( \frac{\pi}{2},\pi)$ (b); $\widetilde{A}<0$ and $\gamma\in (0,\frac{\pi}{2})$ (c);  $\widetilde{A}<0$ and $\gamma\in (\frac{\pi}{2},\pi)$ (d). }\label{fig5}
\end{figure}

From \eqref{qq}, the function $\q$ is 
\begin{equation}\label{q2}
\q=\frac{\kappa(\tau_0)}{\sin\gamma}+\delta \frac{\cot\gamma}{r}=\frac{r\kappa(\tau_0)+\delta\cos\gamma}{r\sin\gamma},\quad 
\delta=\left\{\begin{array}{ll}1,& \mbox{if }A(\nu,\nu)=\frac{1}{r}\\
-1,&\mbox{if }A(\nu,\nu)=-\frac{1}{r}.\end{array}\right.
\end{equation}
The Laplacian operator  written in $(t,s)$-coordinates is $\Delta=\frac{1}{r^2}\partial_{ss}+\partial_{tt}$. Also, $\lvert A\rvert^2=1/r^2$ because $K=0$ and $H=\pm 1/(2r)$ in a cylinder of radius $r$. With respect to the parametrization $X$ defined in \eqref{xx}, we have  $\nabla u=\frac{1}{r^2}u_sX_s+u_t X_t$. Since $\nu=\pm X_s/r$ at the boundary points $s=\pm s_0$, we have   
\begin{eqnarray*}
\frac{\partial u}{\partial\nu}&=\dfrac{u_s}{r},& \mbox{at } s=s_0\\
\frac{\partial u}{\partial\nu}&=-\dfrac{u_s}{r},& \mbox{at }s=-s_0.
\end{eqnarray*}
The eigenvalue problem \eqref{p-e}, with the added condition $u=0$ at $y=0$ and $y=h$ is 
\begin{equation}\label{ei2}
\left.
\begin{aligned}
\frac{1}{r^2}u_{ss}+u_{tt}+(\lambda+\frac{1}{r^2} ) u&=0,\\
u_s(s_0,t)-\frac{r\kappa(\tau_0)+\delta\cos\gamma}{ \sin\gamma}  u(s_0,t)&=0,\\
u_s(-s_0,t)+\frac{r\kappa(\tau_0)+\delta\cos\gamma}{\sin\gamma} u(-s_0,t)&=0,\\
u(s,0)=u(s,h)&=0.
\end{aligned}
\right\}
\end{equation}
 The solutions of \eqref{ei2} are obtained again by the method of separation of variables. Here we follow the same steps discussed in  the system \eqref{ei1} and some parts have been avoided. A general solution $u$ can be expressed as 
$$u(t,s)=\sum_{n=1}^\infty f_n(s)\sin\left(\frac{n\pi t}{h}\right),\quad s\in [-s_0,s_0], t\in [0,h].$$
In order to simplify the notation,   the subscript $n$ for $f_n$ is dropped and  let introduce the notation
 \begin{equation}\label{mm}
 \mu(\tau)=r\kappa(\tau)+\delta\cos\gamma.
 \end{equation}
The first three equations in \eqref{ei2} are 
\begin{equation}\label{b2}
\left.
\begin{aligned}
f''(s) -\left( \frac{r^2n^2\pi^2}{h^2}-\lambda r^2-1\right)f(s)&=0,\\
f'(s_0)-\frac{\mu(\tau_0)}{ \sin\gamma} f(s_0)&=0,\\
f'(-a)+\frac{\mu(\tau_0)}{ \sin\gamma}f(-s_0)&=0.
\end{aligned}
\right\}
\end{equation}
The discussion of solutions is according to the sign of the parenthesis in the first equation of \eqref{b2}. 
\begin{enumerate}
\item Case $ \frac{r^2n^2\pi^2}{h^2}-\lambda r^2-1>0$. Let $\beta>0$, where 
\begin{equation}\label{be}
 \beta^2= \frac{r^2n^2\pi^2}{h^2}-\lambda r^2-1.
 \end{equation}
The general solution of the first equation of \eqref{b2} is $f(s)=Ae^{\beta s}+Be^{-\beta s}$, 
 for some constants $A$ and $B$.  From the boundary conditions \eqref{b2}, we deduce that there are   non-trivial solutions $A$ and $B$ if and only if   
\begin{equation}\label{4t2}
e^{4\beta s_0}=\left(\frac{\beta\sin\gamma + \mu(\tau_0)}{\beta\sin\gamma- \mu(\tau_0) }\right)^2.
\end{equation}
We distinguish two cases. 
\begin{enumerate}
\item Case $\mu(\tau_0)>0$. Let $T= \mu(\tau_0)/\sin\gamma$. Since $\mu(\tau_0)$ is positive,   there is a unique solution $\beta$ of \eqref{4t2}, which belongs to the interval $(T,\infty)$. For this value of $\beta$,  and from \eqref{be},  the eigenvalues are
\begin{equation}\label{be0}
\lambda_n=\frac{n^2\pi^2}{h^2}-\frac{1+\beta^2}{ r^2},\quad n\in\mathbb{N}.
\end{equation}
In particular, if $h$ is sufficiently big, there exist many negative eigenvalues and, consequently, $\Sigma$ is unstable.

 \item Case $\mu(\tau_0)\leq0$. A similar argument as in Sect. \ref{sec3}, there is no  solution of \eqref{4t2}.  
 \end{enumerate}
\item Case $\frac{r^2n^2\pi^2}{h^2}-\lambda r^2-1=0$. The solution of the first equation of \eqref{b2} is $f(s)=A+Bs$, where $A,B\in\r$. The boundary conditions imply that there are non-trivial solutions $A$ and $B$ if and only if
\begin{equation}\label{e43}
\mu(\tau_0)(\sin\gamma-s_0\mu(\tau_0))=0.
\end{equation}
\begin{enumerate}
\item Case $\mu(\tau_0)=0$. Then $B=0$ and   $f(s)=A$.  
\item Case $\mu(\tau_0)\not=0$. Then $s_0= \sin\gamma/\mu(\tau_0)$, in particular, $\mu(\tau_0)>0$. Hence,  $A=0$ and   $f(s)=Bs$. 
\end{enumerate}
In both cases, the eigenvalues are $\lambda_n=-\frac{1}{r^2}+\frac{n^2\pi^2}{h^2}$, $  n\in\mathbb{N}$. Therefore,  $\Sigma$ is unstable if $h$ is sufficiently large.
\item Case $\frac{r^2n^2\pi^2}{h^2}-\lambda r^2-1<0$. Let $\beta>0$ with
$$ \beta^2= -\frac{r^2n^2\pi^2}{h^2}+\lambda r^2+1.$$
The general solution of the first equation of \eqref{b2} is
$$f(s)=A \cos(\beta s)+B\sin(\beta s).$$
 The boundary conditions \eqref{b1} imply that there are non-trivial solutions $A$ and $B$  if and only if 
\begin{equation}\label{4t31}
\tan(\beta s_0)=-\frac{ \mu(\tau_0)}{\beta\sin\gamma},
\end{equation}
or
\begin{equation}\label{4t32} \tan(\beta s_0)=\frac{\beta\sin\gamma}{\mu(\tau_0)}.
\end{equation}
Independently on the sign of $\mu(\tau_0)$, there are   infinitely many solutions of both equations, each one in the intervals of type $I_k$ for every $k\geq 1$. The solution in this interval will be denoted by $\beta_{1,k+1}$ for \eqref{4t31} and $\beta_{2,k+1}$ for \eqref{4t32}. In the interval $I_0=(0,\frac{\pi}{2s_0})$, and only when $\mu(\tau_0)>0$, the root $\beta_{2,1}$ may exist or not. Indeed,  the existence of $\beta_{2,1}$ is assured if and only if $s_0<\sin\gamma/\mu(\tau_0)$. The eigenvalues are 
\begin{equation}\label{be2}
\lambda_{n,k}^j=\frac{n^2\pi^2}{h^2}+\frac{\beta_{j,k}^2-1}{r^2}\quad n,k\in\mathbb{N}, j\in\{1,2\}.
\end{equation}
Since we are studying the stability of $\Sigma$, and because  the cases 1 and 2 prove instability when $\mu(\tau_0)\geq 0$, we only  need to discuss the case $\mu(\tau_0)<0$.  Consider   equation \eqref{4t31}. The function $\beta\mapsto  \tan(\beta s_0)$ defined in $(0,\pi/(2s_0))$ is monotonically increasing from $0$ to $\infty$. The function $\beta\mapsto 1/\beta$ is monotonically decreasing from $\infty $ to $0$ in its domain $(0,\infty)$. Since $\mu(\tau_0)<0$, the first root of   equations \eqref{4t31} and \eqref{4t32} appears by solving \eqref{4t31}, being this root the value  $\beta_{1,1}\in (0,\frac{\pi}{2s_0})$. If $\beta_{1,1}<1$, then \eqref{be2} gives many negative eigenvalues $\lambda_{n,1}^1$ for large values of $h$.  Otherwise, all eigenvalues are nonnegative and the surface is strongly stable because its Morse index is $0$.
  \end{enumerate}
  
  We summarize the above computations in the following theorem.

\begin{theorem}\label{t2}
Let $\s(\c)$ be a symmetric support and let $\Sigma$ be a section of a circular cylinder of radius $r>0$, which it is  a capillary surface on $\s(\c)$ along $L_{s_0}^+\cup L_{s_0}^{-}$. \begin{enumerate}
\item If $\mu(\tau_0)\geq 0$, then $\Sigma$ is unstable.  
\item Suppose $\mu(\tau_0)<0$. If   $\beta_{1,1}<1$, then $\Sigma$ is unstable. Otherwise, $\Sigma$ is  strongly stable.   
\end{enumerate}
\end{theorem}
 
 Theorem \ref{t2} tells us about the stability and instability of $\Sigma$. In case of instability, we study how to calculate the critical value $h_0>0$ in the Plateau-Rayleigh instability criterion. The value $h_0$ is obtained by imposing that there are at least two negative eigenvalues.  Here it is useful to know the distribution of the roots $\beta_{j,k}$. See Fig. \ref{figure-dib}.
 \begin{enumerate}
 \item Case $\mu(\tau_0)<0$. Then $\beta_{1,1}<\beta_{2,1}<\beta_{1,2}<\beta_{2,2}<\cdots$. 
 \item Case $\mu(\tau_0)>0$. Then $\beta_{2,1}<\beta_{1,2}<\beta_{2,2}<\cdots$, where it is assumed that the root $\beta_{2,1}$ exists if $s_0<\sin\gamma/\mu(\tau_0)$. Otherwise, the first root is $\beta_{1,2}$.
\end{enumerate}

\begin{figure}[hbtp]
\begin{center}\scalebox{1}{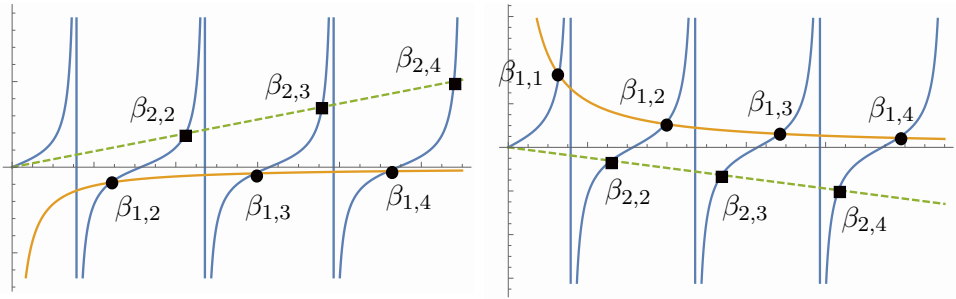 } 
\end{center}
\caption{Distribution of the roots $\beta_{j,k}$ of the equations \eqref{4t31} and \eqref{4t32}. Case $\mu(\tau_0)>0$ (left) and $\mu(\tau_0)<0$ (right), respectively.}\label{figure-dib}
\end{figure}

\begin{corollary}\label{c1}
Let $\s(\c)$ be a symmetric support and let $\Sigma$ be a section of a circular cylinder of radius $r>0$, which it is  a capillary surface on $\s(\c)$ along $L_{s_0}^+\cup L_{s_0}^{-}$. Suppose $\sin\gamma-s_0\mu(\tau_0)\not=0$ and $\mu(\tau_0)\not=0$.   Let $\Sigma_h$ be a rectangular piece of $\Sigma$ of length $h$. If $h>h_0$, then $\Sigma_h$ is unstable, where $h_0$ is given as follows:
\begin{enumerate}
\item Case $\mu(\tau_0)>0$. The value of $h_0$ is   $h_0= \frac{2\pi r}{\sqrt{1+\bar{\beta}_1^2}}$.  
\item Case $\mu(\tau_0)<0$. If $\beta_{1,1}<1$, then  $h_0=\frac{2\pi r}{\sqrt{1-\beta_{1,1}^2}}$.

\end{enumerate}
\end{corollary}

\begin{proof} 
The discussion of the case $\mu(\tau_0)<0$ is a consequence of (2) of Thm.  \ref{t2}. Suppose now  $\mu(\tau_0)>0$. For $\beta_1$, we have $\lambda_2<0$ if and only if $h>2\pi r/\sqrt{1+\beta_1^2}$. If all roots of \eqref{4t31} and \eqref{4t32} are greater than or equal to  $1$, then $\lambda_{n,k}^j\geq 0$ and the result is proved. Otherwise, if $\beta_{j,k}$ is the first root of  \eqref{4t31} or \eqref{4t32} with $\beta_{j,k}<1$, then $\lambda_{2,k}^j<0$ if and only if  $h>\frac{2\pi r}{\sqrt{1-\beta_{j,k}^2}}$. However, $ 2\pi r/\sqrt{1-\beta_{j,k}^2}> 2\pi r/\sqrt{1+\beta_1^2}$.   
\end{proof}

In the following examples, we show  how to find the value $h_0$ in the Plateau-Rayleigh instability criterion using Corollary \ref{c1}. In these examples, the support surface is a parabolic  cylinder considering the cases that $\kappa$ is positive and negative. It will be assumed that the contact angle is $\gamma=\pi/2$.

\begin{example}\label{p1}
{\rm  Let $\s(\c)$ be a parabolic cylinder generated by the parabola $\c(\tau)=(\tau,0,\tau^2)$.  Let $\gamma=\pi/2$. We now calculate a section of circular cylinder $\Sigma$ intersecting  orthogonally $\s(\c)$.   Fixed  $\c(\tau)$, the   circle  $\R$ orthogonal to $\c$ at $\tau$ is centered at $(0,0,-\tau^2)$ and of radius $r=\tau\sqrt{1+4\tau^2}$.  The circle $\R$ is the generating curve of a circular cylinder $\Sigma$ of radius $r$. The two points $\c(\tau)$ and $\c(-\tau)$ define two arcs in the circle $\R$. If the section of the circular cylinder is contained in the convex domain determined by $\s(\c)$, the curvature $\kappa$ is positive (Fig. \ref{fig6}, left); for the other section, $\kappa$ is negative, see Ex. \ref{p3} below (Fig. \ref{fig6}, right).  

Suppose the first case. We know by Thm.  \ref{t2} that $\Sigma$ is unstable for large values of $h$. We fix the contact  at $\tau_0=1$. Then $L_{1}^+\cup L_1^{-}$ are the straight-lines  $\{x=\pm 1,z=1\}$.  Now $r=\sqrt{5}$, $\kappa(1)=1/5^{3/2}$  and 
$\mu(1)= 2/5$.  The value $s_0$ such that $\R(s_0)=\c(1)$ is $s_0=\sin^{-1}(1/r)= 0.463$. Notice that $\mu(1)>0$ and $\sin\gamma-\mu(1) s_0\not=0$. This implies that equation \eqref{e43} has no solutions.
We   solve \eqref{4t2}, obtaining $\beta_1= 0.958$.   By Cor.  \ref{c1},  we have $h_0=2\pi r/\sqrt{1+\beta_1^2} =10.142$.
}
\end{example}

\begin{example}\label{p2}
{\rm   We study other circular cylinders that are capillary with the above parabolic cylinder $\s(\c)$. The purpose is to compare the value $h_0$ after solving equations \eqref{4t2}, \eqref{4t31} and \eqref{4t32}. Note that $s_0=\sin^{-1}(\tau_0/r)$. 
\begin{enumerate}
\item Let $\tau_0=1/4$. Then $s_0=1.107$. The solution of \eqref{4t2} is $\beta_1=0.649$ and thus  $h_0=2\pi r/\sqrt{1+\beta_1^2}= 1.473$.   
\item Let $\tau_0=2$. Then $s_0=0.244$ and  the solution of \eqref{4t2} is $\beta_1=0.989$ and   $h_0=36.829$.
\item Let $\tau_0=4$. Now $s_0=0.124$ and the solution of \eqref{4t2} is $\beta_1=0.997$. Then   $h_0=2\pi r/\sqrt{1+\beta_1^2}=143.466$.
\end{enumerate} 
}
\end{example}

\begin{example}\label{p3}
{\rm  Let $\s(\c)$ be the same parabolic cylinder,   but now we assume that the arc of $\R$ that determines $\Sigma$ lies outside of the convex domain determined by $\c$ (Fig. \ref{fig6}  right). Then $\kappa$ is negative. Suppose again $\gamma=\pi/2$, in particular,  
$$\mu(\tau)=-\frac{2\tau}{1+4\tau^2}.$$
Consider the case $\tau_0=2$, or equivalently, $s_0=0.244$. A computation of the solutions of \eqref{4t32} gives $\beta_{1,1}=0.970$, which is less than $1$. The critical value   is  $h_0=2\pi r/\sqrt{1-\beta_{1,1}^2}=215.687$: see (2) of Cor.  \ref{c1}.}
\end{example}

\begin{figure}[hbtp]
\begin{center}\scalebox{1}{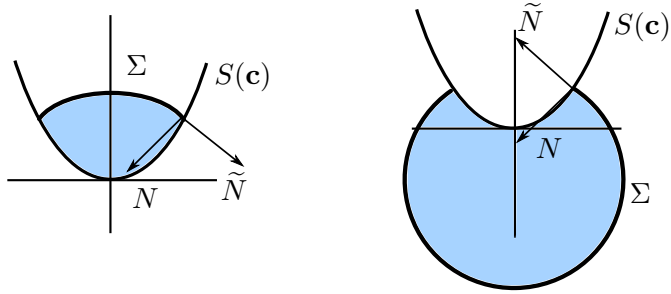 }
\end{center}
\caption{A circular cylinder as capillary surface on a parabolic cylinder $\s(\c)$. Left: concave case (Exs. \ref{p1} and \ref{p2}). Right: convex case (Example \ref{p3}).}\label{fig6}
\end{figure}
%%%%%%%%%%%%%%%%%%
\section{Bifurcating liquid channels}\label{sec5}
%%%%%%%%%%%%%%

In this section we investigate the existence of  new examples of capillary surfaces  using  techniques of bifurcation. The methods  are based on the  ``bifurcation at a simple eigenvalue'' of Crandall and Rabinowitz \cite{cr}. We will recall how works the method  but applied to our context of capillary surfaces. Let $\Sigma$ be a capillary circular cylinder of radius $r_0>0$ on a symmetric support $\s(\c)$. Let $H_0=\pm 1/(2r_0)$ be its mean curvature, where the sign depends on the orientation on $\Sigma$. 

First it is necessary to have a one-parameter family of capillary surfaces containing $\Sigma$ as one of its members, and this family will be perturbed to find the new examples. These surfaces are again circular cylinders making the same contact angle with the support surface. The construction of these cylinders is as follows.  Let  $s_0\in I$ be the contact point of the circle $\R$ with  $\c$, $\R(s_0)=\c(\tau_0)$. This circle is centered at a point $m_0$ of the $z$-axis and of radius $r=\lvert\c(s_0)-m_0\rvert$. Let $v$ be a unit vector in the direction of $\c(\tau_0)-m_0$. If  $\textbf{n}$   denotes the unit normal vector of $\c$, the contact angle $\gamma$ satisfies $\cos\gamma=\langle \textbf{n}(\tau_0),v\rangle$. Around a neighborhood $(\tau_0-\epsilon,\tau_0+\epsilon)$ of $\tau_0$, we can change smoothly $v$ by   unit vectors  $v(\tau)$ in the $xz$-plane such that $\langle \textbf{n}(\tau),v(\tau)\rangle=\cos\gamma$. Each vector $v(\tau)$ determines a straight-line through $\c(\tau)$ and direction $v(\tau)$. Its intersection with the $z$-axis defines a center and a radius of a circle. The circular cylinder generated by this circle   is a capillary surface on $\s(\c)$   with contact angle $\gamma$.

\begin{remark}{\rm  
The above construction does not work for   capillary planar strips.  Fixed a plane $\Pi$ making an angle $\gamma$ with $\s(\c)$, any plane close to $\Pi$ must be parallel in order to be a capillary surface (Example \ref{ex1}). However, the contact angle is different and we have no a family of capillary planar strips making the same contact angle.  }
\end{remark}

  Let $V$ be an open set of $0\in C^\infty(\Sigma)$ such that  for all $u\in V$,   $\Sigma_u=\{p+u(p)N(p)\in\r^3:p\in\Sigma\}$ is an embedded surface.   If $H(u)$ is the mean curvature of the surface $\Sigma_u$, define 
\begin{equation}\label{ff}
F\colon   V\times\r\to C^\infty(\Sigma),\quad F(u,H)=2(H-H(u)).
\end{equation}
Therefore, $F(u,H)=0$ if and only if $\Sigma_u$ is a surface with constant mean curvature equal to $H$. Since the mean curvature of $\Sigma$ is $H_0$, we have $F(0,H_0)=0$. On the other hand,  we parametrize the above family of circular cylinders by $H$. Around the value $H_0$, each cylinder is the normal graph $\Sigma_{w_H}$ for a certain function $w_H\in C^\infty(\Sigma)$. Then $F(w_{H},H)=0$, where  $w_{H_0}=0$.   These cylinders $\Sigma_{w_H}$  can be interpreted as  a curve $H\mapsto (w_H,H)$ of solutions of \eqref{ff} through $H_0$. The objective in our result of bifurcation is to perturb this family of solutions around $(0,H_0)$  by means of normal variations to get  another curve of solutions of \eqref{ff}  through the same point $(0,H_0)$. 

The uniqueness of solutions of $ F(u,H)=0$  is related with the Implicit Function theorem. If the Fr\'echet derivative $D_uF$ at $(0,H_0)$ is bijective, then there is $\delta>0$,  a neighborhood of $0\in C^\infty(\Sigma)$, which we suppose to be $V$ again, and a unique injective map $\Psi\colon (H_0-\delta,H_0+\delta)\to V$ with $\Psi(H_0)=0$ and $F(\Psi(H),H)=0$ for all $H\in (H_0-\delta,H_0+\delta)$. By uniqueness, $\Psi(H)=w_H$ and the family   $\Sigma_{ \Psi(H)}$ coincides with the above  one-parameter family of  circular cylinders   $\Sigma_{w_H}$.

Because we are searching   examples different from the cylinders $\Sigma_{w_H}$, the Implicit Function Theorem must fail. The Fr\'echet derivative $D_uF(0,H_0)$ is
$$D_u F(0,H_0)[v]=-\mathcal{L}[v], \quad v\in C^\infty(\Sigma),$$
where $\mathcal{L}$ is the Jacobi operator \eqref{l1}  \cite{ko,vo}. Thus if uniqueness fails, the equation $\mathcal{L}[v]=0$ has non-trivial solutions. This implies that its kernel is not trivial, that is,  $0$ is an eigenvalue of $\mathcal{L}$.

 We now state the   result of bifurcation of   Crandall and Rabinowitz     \cite{cr}. See also \cite[Th. 13.5]{smo}. In its abstract formulation, the result says the following.

\begin{theorem}\label{t-bifu} Let $X$ and $Y$ be  Banach spaces, $V$ an open subset of $X$ and $\mathcal{F}:V\times I\rightarrow Y$ be a twice continuously Fr\'echet differentiable functional, where $I\subset \r$ and $H_0\in I$.  Let $w\colon (H_0-\epsilon,H_0+\epsilon)\to V$ be a differentiable map such that $w(H_0)=0$ and  $\mathcal{F}(w(H),H)=0$ for   $H\in (H_0-\epsilon,H_0+\epsilon)$. Suppose:
\begin{enumerate}
\item $\mbox{dim Ker}(D_u \mathcal{F}(0,H_0))=1$. Assume that Ker$(D_u \mathcal{F}(0,H_0))$ is spanned by $u_0$.
\item The codimension of  $\mbox{Im } D_u \mathcal{F}(0,H_0) $ is $1$.
\item $D_HD_u\mathcal{F}(0,H_0)(u_0)\not\in\mbox{Im }D_u \mathcal{F}(0,H_0) $.
\end{enumerate}
Then  there exists a   continuously differentiable curve $s\mapsto (u(s),H(s))$, $s\in(H_0-\epsilon,H_0+\epsilon)$, with $u(0)=0$, $H(0)=H_0$, such that $\mathcal{F}(u(s),H(s))=0$, for any $|s|<\epsilon$. Moreover, $(0,H_0)$ is a bifurcation point of the equation $\mathcal{F}(u,H)=0$ in the following sense: in a neighborhood  of $(0,H_0)$, the set of solutions of $\mathcal{F}(u,H)=0$ consists only of the curve $H\mapsto (w_H,H)$ and the curve $s\mapsto (u(s),H(s))$.
\end{theorem}

The arguments of our main result in this section (Thm.  \ref{t4} below) follow the same procedure in \cite{vo}, where Vogel gets bifurcation of circular cylinders supported on wedges.  We refer to Sect. 3 of  \cite{vo} for further details.
 Since  $\Sigma$ is an infinite circular cylinder, it is natural to  impose periodicity in the boundary conditions along the direction of the rulings. Let $\Sigma$ be parametrized  by \eqref{pc}, where $(t,s)\in \r\times [-s_0,s_0]$.  By this periodicity,  $\Sigma$ is identified with the rectangle $ [0,2\pi]\times [-s_0,s_0]$ and $\partial\Sigma$ with $  [0,2\pi]\times \{-s_0,s_0\}$ in the $(t,s)$-domain. For further purposes, we need to restrict to   functions   that are even functions in the $t$ variable.

 Let $V$ be an open set of $C^{\infty}_{2\pi,even}(\Sigma)$, $0\in V$,    such   that $\Sigma_u$ is a normal graph on $\Sigma$ for all $u\in V$. Here  the subscript means that the functions are $2\pi$-periodic and even on the $t$-variable. Define
$$\mathcal{F}:V\times\r\rightarrow  C^{\infty}_{2\pi,even}(\Sigma)\times C^{\infty}_{2\pi,even}(\partial \Sigma) ,\quad  \mathcal{F}(u,H)=(F(u,H),\gamma(u)-\gamma),$$
where $\gamma(u)$ is the contact angle of  $\Sigma_u$ with $\s(\c)$. The regularity of   $\mathcal{F}$ is assured by a standard argument (\cite[Lem. 3.7, 3.8]{vo}).  The partial of $D_u\mathcal{F}$ at $(0,H)$   is  
 \begin{equation}\label{18}
 L[v]:=D_u\mathcal{F}(0,H)[v] =(\mathcal{L}[v],\mathbf{b}(v)).
 \end{equation}
 Here  $\mathbf{b}(v)= \frac{1}{r}v_s\mp  \q v$, where the  sign $-$ (resp. $+$)  corresponds with the value at $  [0,2\pi]\times \{s_0\}$ (resp.   $ [0,2\pi]\times \{-s_0\}$)  \cite{vo,we}.

According to Thm.  \ref{t4}, the first step is to calculate the eigenspace $E_0$ for $L$.   We need to come back to the  eigenvalue problem \eqref{ei2} using separation of variables. In contrast to the proof of  Thm.  \ref{t2}, we cannot impose the last boundary condition in \eqref{ei2}  in the walls $y=0$ and $y=h$ because of the $2\pi$-periodicity on the variable $t$.  Thus, let $u$ expand as a Fourier expression 
$$u(t,s)=\sum_{n=0}^\infty f_n(s)\cos(nt)+\sum_{n=1}^\infty g_n(s)\sin(nt).$$
The first equation in \eqref{ei2} is
$$\sum_{n=0}^\infty \left(\frac{1}{r^2}f_n''+\left(\frac{1}{r^2}+\lambda-n^2\right)f_n\right)
\cos(nt)+\sum_{n=1}^\infty \left(\frac{1}{r^2}g_n''+\left(\frac{1}{r^2}+\lambda-n^2\right)g_n\right) \sin(nt)=0.$$
The boundary conditions in \eqref{ei2}  are 
$$\sum_{n=0}^\infty \left(f_n'(s_0)-\frac{\mu}{\sin\gamma}f_n(s_0)\right)\cos(nt)+
\sum_{n=1}^\infty \left(g_n'(s_0)-\frac{\mu}{\sin\gamma}g_n(s_0)\right)\sin(nt)=0,$$
$$\sum_{n=0}^\infty \left(f_n'(-s_0)+\frac{\mu}{\sin\gamma}f_n(-s_0)\right)\cos(nt)+
\sum_{n=1}^\infty \left(g_n'(-s_0)+\frac{\mu}{\sin\gamma}g_n(-s_0)\right)\sin(nt)=0,$$
We need to look for non-trivial solutions of 
\begin{equation}\label{b4}
\left.
\begin{aligned}
f_n''(s)+ \left(1+\lambda r^2-n^2 r^2\right)f_n(s)&=0, \\
f_n'(s_0)-\frac{\mu(\tau_0)}{\sin\gamma}f_n(s_0)&=0, \\
f_n'(-s_0)+\frac{\mu(\tau_0)}{\sin\gamma}f_n(-s_0)&=0, 
\end{aligned}
\right\}
\end{equation}
and analogously for the functions $g_n$. If   a function $u$ is a non-trivial solution of $\mathcal{L}[u]=0$, then   $f_n$ and $g_n$ are non-zero for some natural numbers $n\in\mathbb{N}$. Notice that   equations \eqref{b4} for $f_n$ and $g_n$ coincide.

 The solution of \eqref{b4} depends on the sign of $1+\lambda r^2-n^2 r^2$.  As we noticed at the beginning of this section, for bifurcation, the first eigenvalue of $L$ must be negative because otherwise, $\Sigma$ would be strongly stable. If $\lambda=0$ is an eigenvalue, then the parenthesis in the first equation in \eqref{b4} is $1-n^2 r^2$.  The arguments that follow would be difficult to carry out due to the presence of the two unknowns $n$ and $r$ in \eqref{b4}.  For that reason, it will be assumed  that $0$ is the  second eigenvalue of $L$ ($n=1$ in \eqref{b4}). This reduces the work to find the value of $r$, that is,  the cylinder where bifurcation appears. 
  
    We  now repeat the same arguments done in Sect. \ref{sec4}.
\begin{enumerate}
\item Case $1+\lambda r^2-n^2 r^2>0$. Let $\beta>0$ with $\beta^2=1+\lambda r^2-n^2 r^2$. The eigenvalues are 
\begin{equation}\label{ev1}
\lambda_n=n^2+\frac{\beta^2-1}{r^2},\quad n\in\mathbb{N}.
\end{equation}
The first eigenvalue corresponds with $n=0$ and $\lambda_0<0$ implies $\beta<1$. The condition $\lambda_1=0$ yields $\beta=\sqrt{1-r^2}$. The solution of \eqref{b4} is $f_1(s)=A\cos(\beta s)+B\sin(\beta s)$, and from \eqref{4t31} or \eqref{4t32} we have  
\begin{equation}\label{v1}
\tan(\beta s_0)=-\frac{\mu(\tau_0)}{\beta\sin\gamma}\quad\mbox{or}\quad \tan(\beta s_0)=\frac{\beta\sin\gamma}{\mu(\tau_0)}.
\end{equation}
In the first case, $B=0$ and $f_1(s)=A\cos(\beta s)$. In the second one, we have $A=0$ and $f_1(s)=B\sin(\beta s)$. For further purposes, it deserves to observe that in both cases,  we have $f_1(s_0)^2=f_1(-s_0)^2$. We study a little more the solutions of \eqref{v1} according to the sign of $\mu(\tau_0)$. If $\mu(\tau_0)=0$, then   there are no solutions of \eqref{v1}.
\begin{enumerate}
\item Case $\mu(\tau_0)>0$. 
 There are many solutions     of the first equation \eqref{v1} and the first root belongs to $(\frac{\pi}{2s_0},\frac{3\pi}{2s_0})$.   The second equation of \eqref{v1} also has many solutions and the first root belongs to the interval $(0,\frac{\pi}{2s_0})$ if $s_0<\sin\gamma/\mu(\tau_0)$ or to the interval $(\frac{\pi}{2s_0},\frac{3\pi}{2s_0})$ if $s_0>\sin\gamma/\mu(\tau_0)$.  
 
\item Case $\mu(\tau_0)<0$.
  There are many solutions of     the first equation \eqref{v1} and the first root belongs to $(0,\frac{\pi}{2s_0})$.  The second equation of \eqref{v1} also has many solutions and the first root belongs to the interval $(\frac{\pi}{2s_0},\frac{3\pi}{2s_0})$.  
   \end{enumerate}
\item Case $1+\lambda r^2-n^2 r^2=0$. 
 The eigenvalues are 
\begin{equation}\label{ev2}
\lambda_n=n^2-\frac{1}{r^2},\quad n\in\mathbb{N}.
\end{equation}
Now $\lambda_0=-1/r^2$  is negative and $\lambda_1=1-\frac{1}{r^2}$. The condition $\lambda_1=0$ gives $r=1$. If $\mu(\tau_0)=0$, then $f_1(s)=A$. If $\mu(\tau_0)\not=0$, then \eqref{e43} implies
\begin{equation}\label{v2}
s_0=\frac{\sin\gamma}{\mu(\tau_0)}.
\end{equation}
 Now $A=0$ and $f_1(s)=Bs$. In both cases, $f_1$ satisfies  $f_1(s_0)^2=f_1(-s_0)^2$.

\item Case $1+\lambda r^2-n^2 r^2<0$.  Let $\beta>0$ with $\beta^2=-1-\lambda r^2+n^2 r^2$. The eigenvalues are 
\begin{equation}\label{ev3}
\lambda_n=n^2-\frac{1+\beta^2}{r^2},\quad n\in\mathbb{N}.
\end{equation}
Then $\lambda_0=-\frac{1+\beta^2}{r^2}<0$. Letting $\lambda_1=0$, we have  $\beta=\sqrt{r^2-1}$. By \eqref{4t2},  it is required that $\beta$ is a solution of
\begin{equation}\label{v3}
e^{4\beta s_0}=\left(\frac{\beta\sin\gamma+\mu(\tau_0)}{\beta\sin\gamma-\mu(\tau_0)}\right)^2.
\end{equation}
We now discuss  the solutions of \eqref{v3} depending on the sign of $\mu(\tau_0)$.
\begin{enumerate}
\item Case $\mu(\tau_0)>0$.  There is a unique solution $\beta$ of \eqref{v3}, with $\beta>\mu(\tau_0)/\sin\gamma$.  
\item Case $\mu(\tau_0)\leq 0$. There are no solutions of \eqref{v3} following the same arguments as in  Sect. \ref{sec4}. 
\end{enumerate}
If such a solution exists, then 
$$f_1(s)=A\left(e^{\beta s}+\left(\frac{\beta\sin\gamma-\mu(\tau_0)}{\beta\sin\gamma+\mu(\tau_0)}\right)e^{\beta (2s_0-s)}\right).$$
Using \eqref{v3}, the function $f_1$ is 
$$f_1(s)=A(e^{\beta s}+e^{-\beta s})=2A\cosh(\beta s).$$
Again,  the function $f_1$ satisfies $f_1(s_0)^2=f_1(-s_0)^2$.
\end{enumerate}
Notice that the three cases are mutually disjoint because  $r<1$ in (1), $r=1$ in (2)  and $r>1$ in (3).   It is clear that in (2) and (3) the solution is unique. It is only in the case (1) where we may have many solutions. Recall that in (1) we are imposing   $\beta<1$.

Suppose that there exists    only one solution $\beta$ of \eqref{v1}, \eqref{v2} or \eqref{v3}. For this value of $\beta$,   we find an eigenfunction $f_1(s)\cos(t)$ for the eigenvalue $0$. Similarly, we have $f_1(s)\sin(t)$. Since we restrict to even functions in the $t$-variable, the function $f_1(s)\sin(t)$ is discarded.  Then   $E_0$ is of dimension $1$ and spanned by $u_0(t,s)=f_1(s)\cos(t)$. Consequently, it holds (1) of Thm.  \ref{t-bifu}.  

\begin{remark} {\rm Once we have proved that $E_0$ is one-dimensional, we can observe that $\int_\Sigma u_0\, d\Sigma=0$  because $u_0(t,s)=f_1(s)\cos t$ and $\int_0^{2\pi}\cos t\, dt=0$. This is expected since, otherwise, there is uniqueness of solutions of \eqref{ff} around $(0,H_0)$. Indeed, it was proved in \cite{ko} that if  the eigenspace $E_0$ of  the eigenvalue $0$ is of dimension $1$, with $E_0=\mbox{span}(u_0)$ and if $\int_\Sigma u_0\, d\Sigma\not=0$, then there is a unique deformation of $\Sigma$ by  $H$-surfaces, where $H$ belongs to a neighborhood of $H_0$. }
\end{remark}

We now turn to study the condition (2) of Thm.  \ref{t-bifu} calculating $\mathrm{Im}(L)$, where $L$ is defined in \eqref{18}, and showing that its codimension is $1$. The argument is standard and it  is given in \cite[Lem. 3.12]{vo} after the needed modifications for  our context.

\begin{lemma}\label{le-vo} Suppose that $L$ has a kernel spanned by $u_0$. A pair $(h_1,h_2)\in C^{\infty}_{2\pi,even}(\Sigma)\times C^{\infty}_{2\pi,even}(\partial\Sigma)$ belongs to $\mathrm{Im}(L)$ if and only if 
\begin{equation}\label{oo}
\int_0^{2\pi}\int_{-s_0}^{s_0} u_0h_1\, ds\ dt=
 -2H\int_0^{2\pi}\left(u_0(-s_0,t)h_2(-s_0,t)+u_0(s_0,t)h_2(s_0,t)\right)\,dt.
 \end{equation}
 \end{lemma}
 Using the divergence theorem in the right hand-side of \eqref{oo}, we have    $(h_1,h_2)\in \mathrm{Im}(L)$ if and only if 
 $$\int_0^{2\pi}\int_{-s_0}^{s_0}  u_0(h_1-\mathcal{L}[g])=0.$$
 This proves that the image of $L$ can be written as an $L^2$-orthogonal complement  to $u_0$, proving that its codimension is $1$.

 Finally, we investigate the condition (3) of Thm.  \ref{t-bifu}. We replace the radii $r$ of   cylinders in terms of the mean curvatures. Recall that $H=\pm 1/(2r)$ depending on the orientation of $\Sigma$. To simplify the notation, we will write $H=1/(2r)$ understanding a possible change of sign.  From \eqref{qq} and \eqref{l1}, the expression of $L$ is 
$$
 L[u]=\left(4H^2 u_{ss}+u_{tt}+4H^2 u, 2Hu_s \mp\left(\frac{\kappa}{\sin\gamma}+2H\cot\gamma\right)u\right).
$$
Therefore, 
 \begin{equation}\label{duh}
 D_HL[u]=(8H(u_{ss}+u), 2  u_s \mp 2 \cot\gamma u).
\end{equation} 
We know that  $u_0(t,s)=f_1(s)\cos(t)$, where $f_1(s)$ has been obtained in the previous discussion of the three cases.  First let us compute $D_HL[u_0]$ in \eqref{duh}. Let  $1+\lambda r^2-n^2 r^2=\epsilon\beta^2$, where $\epsilon$ is $1$, $0$ or $-1$  depending on the above cases (1), (2) and (3), respectively. Then \eqref{duh} yields
\begin{equation*}
\begin{split}
D_HL[u_0]&=(8H((u_0)_{ss}+u_0,2(u_0)_s\mp \cot\gamma u_0)\\ 
&=\left(8H((f_1)_{ss}+f_1) ,2((f_1)_s) \mp \cot\gamma f_1\right)\cos(t) \\
&=\left(8H(1-\epsilon \beta^2)f_1,2((f_1)_s) \mp \cot\gamma f_1\right)\cos(t),
\end{split}
\end{equation*}
Let us use Lemma \ref{le-vo} and the above expression of $D_HL[u_0]$. The left-hand side of \eqref{oo} is 
\begin{equation}\label{xxx}
8H(1-\epsilon\beta^2)\int_{-s_0}^{s_0}f_1(s)^2\, ds\int_0^{2\pi}\cos^2(t)\, dt.
\end{equation}
 For the right-hand side   of \eqref{oo}, it is  used the boundary conditions that satisfies $f_1$, obtaining
\begin{equation}\label{xxxx}
\begin{split}
&-4H\left(f_1(-s_0)(f_1'(-s_0)+\cot\gamma f_1(-s_0))+f_1(s_0)(f_1'(s_0)-\cot\gamma f_1(-s_0))\right)\int_0^{2\pi}\cos^2(t)\, dt\\
=&-4H\ \left(f_1(-s_0)^2(-\frac{\mu(-\tau_0)}{\sin\gamma}+\cot\gamma )+f_1(s_0)^2(\frac{\mu(\tau_0)}{\sin\gamma}-\cot\gamma )\right)\int_0^{2\pi}\cos^2(t)\, dt\\
=& -4H (f_1(s_0)^2-f_1(-s_0)^2)\frac{\mu(\tau_0)}{\sin\gamma} \int_0^{2\pi}\cos^2(t)\, dt\\
=&\, 0,
\end{split}
\end{equation}
where we have used $\mu(-\tau_0)=\mu(\tau_0)$ and in the last identity, that  $f_1(s_0)^2=f_1(-s_0)^2$. From  \eqref{xxx} and \eqref{xxxx}, it is deduced that  $D_HL[u_0]\not\in\mathrm{Im}(L)$ if and only if  
\begin{equation}\label{e3}
(1-\epsilon\beta^2)\int_{-s_0}^{s_0}   f_1(s)^2\, ds\not=  0.
\end{equation}
This is equivalent to $\beta\not=1$. Therefore, we need to come back to the solutions of \eqref{v1}, \eqref{v2} and \eqref{v3} and check this condition. For \eqref{v1}, it was required that $\beta<1$ in order to ensure instability of the surface, in particular, $\beta\not=1$. For the solutions of \eqref{v2} we have $\epsilon=0$, so \eqref{e3} is verified.  As a conclusion,  we obtain the next result of bifurcation.

 \begin{theorem}\label{t4}
 Let $\Sigma$ be a section of a circular cylinder which it is a capillary surface on $\s(\c)$ with contact angle $\gamma$. If there is a unique solution $\beta$ of equations \eqref{v1}, \eqref{v2} and \eqref{v3}, with $\beta\not=1$, then $0$ is a simple eigenvalue of $L$ and consequently  there is a one-parameter family of $2\pi$-periodic capillary surfaces on $\s(\c)$ with contact angle $\gamma$ which bifurcate from $\Sigma$.  
 \end{theorem}

 In fact, returning to the arguments, the solution of \eqref{v1}, \eqref{v2} and \eqref{v3} is unique except if there are many solutions (and less than $1$) for \eqref{v1}.  Another observation   worth mentioning is the choice $n=1$ in  \eqref{b4}. With this choice we have obtained $2\pi$-periodic surfaces in the $t$-direction. However, another natural number $n\in\n$ could have been chosen. The difference is that now the surface would be $2\pi/n$-periodic in the $t$-variable. Anyway, the new family of surfaces obtained from the bifurcation with this value $n\in\n$ would be the same of Thm.  \ref{t4} after up scaled by the factor $n$.

Once we have arrived at this stage, the problem to prove that $0$ in an eigenvalue of $L$  is difficult to address   in all   generality of support surfaces and  contact  angles $\gamma$.   In the next section, we will apply Thm.  \ref{t4} in some explicit examples.

%%%%%%%%%%%%%%%%%%%%%%%%%%%%%%%%%%%%%%

\section{Bifurcation in explicit examples of support surfaces}\label{sec6}
%%%%%%%%%%%%%%%%%%%%%%%%%%%%%%%%%%%%%%

We  study in this section the bifurcation of Thm.  \ref{t4} in   particular supports. By the nature of equations   \eqref{v1}, \eqref{v2} and \eqref{v3}, the computations are hard to carry out. Here we restrict to the case that the contact angle is   $\gamma=\pi/2$. This situation of free boundary case  has its own interest because it represents the situation $\sigma_{SG}=\sigma_{SL}$ for  the surface tensions.  The contact angle $\pi/2$  is expected because it means that the solid behaves equally for the air and liquid phases. This scenario has been already studied in the literature for the case that the support is a horizontal plane or a circular cylinder (see references in the Introduction). The novelty of the present paper is that $\s(\c)$ is of arbitrary shape, which it is has not been previously considered. The purpose of this section  is to determine explicitly the cylinder $\Sigma$ of Thm.  \ref{t4} which it is  used in the bifurcation techniques.

Now we have   $\sin\gamma=1$ and  $\mu(\tau_0)=r\kappa(\tau_0)$. The supports to investigate are those whose generating curves $\c$ are:
\begin{enumerate}
\item A parabola  and $\kappa>0$. 
\item A parabola  and   $\kappa<0$.
\item A catenary and  $\kappa>0$.
\end{enumerate}
In the first and third cases, we will prove the existence of a new family of capillary surfaces bifurcating from a circular cylinder, whereas in the second case, we will see that $0$ is not an eigenvalue of $L$. The situation when the support surface is a horizontal plane or a wedge was studied  by the author in \cite{lo1,lo2}. In both references (or now using Thm.  \ref{t4}), it is possible to give a mathematical proof of the bifurcation phenomenon of  several   experiments carried out at the Max Planck Institute of Colloids and Interfaces (MIPKG), at Potsdam.   These experiments showed the formation of budges after spreading liquid  on microchannels formed alternatively by   hydrophilic and hydrophobic strips and of long liquid channels supported on a wedge \cite{bkk,ghll,kmlr,li}. 

 First we describe the two steps in the process.
\begin{enumerate}
\item Calculate the one-parameter family of capillary circular cylinders with contact angle $\gamma=\pi/2$. 

 \item Solve  \eqref{v1}, \eqref{v2} and \eqref{v3} and check the uniqueness of solutions. 
\end{enumerate}

%%%%%%%%%%%%%%%%%%%%%%%%%%%%%%%%%%%%%
\subsection{The support surface is a parabolic cylinder  and $\kappa>0$}
%%%%%%%%%%%%%%%%%%%%%%%%%%%%%%%%%%%

 The parametrization of the parabola is $\c(\tau)=(\tau,0,\tau^2)$. Since   $\gamma=\pi/2$,   the circle $\R$ generating the  cylinder is centered at $(0,0,-\tau^2)$ and its radius  is $r=\tau\sqrt{1+4\tau^2}$. Because $\kappa>0$, we are taking the section  of circular cylinder  contained in the convex domain of $\r^3$ determined by $\s(\c)$: see Fig. \ref{fig6} left.  Using \eqref{k}, we know $\kappa=2/(1+4\tau^2)^{3/2}$. On the other hand, the relation between $\tau_0$ and $s_0$ is $\sin{s_0}=\tau_0/r$. 
 We   solve \eqref{v1}, \eqref{v2} and \eqref{v3}. We will see that in the discussion of \eqref{v3}, cylinders do bifurcate. 

\begin{enumerate}
\item Solutions of \eqref{v1}. We know that $\beta<1$ and   $r=\sqrt{1-\beta^2}$.  From the   value of $r$, the condition $\beta<1$ implies that $\tau_0<0.624$. Using this information, equations \eqref{v1}   have no  a solution and thus this case is not possible.

\item Solutions of \eqref{v2}. We have $r=1$ and  then  $\tau_0 = 0.624$. Hence, $s_0=0.674$. However  this value $s_0$ does not satisfy \eqref{v2} because $1/\mu(\tau_0)=2.049$. This case is not possible. 

\item Solutions of \eqref{v3}. Since $r>1$, then $\tau_0>0.624$. We know that  $r=\sqrt{\beta^2+1}$.  By solving  \eqref{v3}, we obtain the solution $\tau_0= 0.754$, or equivalently, $s_0=0.585$. Using this value, then $r= 1.364$ and $\beta= 0.928$. For these values of $r$ and $\beta$, bifurcation occurs according to  Thm.  \ref{t4} is of radius $1.364$ and  the center of the generating circle $\R$ is $(0,0,-0.568)$.  The function $u_0$ is $u_0(t,s)=\cosh(\beta s)\cos t$ and \eqref{e3} holds because $\beta\not=1$.
\end{enumerate}
%%%%%%%%%%%%%%%%%%%%%%%%%%
\subsection{The support surface is a parabolic cylinder  and $\kappa<0$}
%%%%%%%%%%%%%%%%%%%%%%%%%%%%%%%%%%%

We now consider the case that the support surface is a parabola but assuming that $\kappa$ is negative. This means that the fluid lies below $\s(c)$, see Fig. \ref{fig5} right. We will show that $0$ is not an eigenvalue of $L$ and, consequently, there is not  bifurcation by simple eigenvalues.   The circular cylinder is  the same as in the case $\kappa>0$ but considering the section of this   cylinder which lies below $\s(\c)$. Now      $\kappa=-2/(1+4\tau^2)^{3/2}$ and  $\mu=r\kappa=-2\tau/(1+4\tau^2)$.

\begin{enumerate}
\item Solutions of \eqref{v1}.   The condition   $\beta<1$ implies   $\tau_0< 0.624$. However both equations \eqref{v1} have no   solutions. 
\item Solutions of \eqref{v2}.     Since $\mu<0$, there is no positive solution of \eqref{v2}.   This case is not possible. 

\item Solutions of \eqref{v3}.    Now $0$ is an eigenvalue of $L$ if $r=\sqrt{ \beta^2+1}$, where $\beta$ is a solution of \eqref{v3}. The condition $\beta>1$ implies that  $\tau_0> 0.624$. Equation   \eqref{v3} has no solution and this case is not possible.  
 
\end{enumerate}

%%%%%%%%%%%%%%%%%%%%%%%%%%%%%%%%%%%%%%%%%%%
\subsection{The support surface is a catenary cylinder and $\kappa>0$}
%%%%%%%%%%%%%%%%%%%%%%%%%%%%%%%%%%%

 The parametrization of the catenary is $\c(\tau)=(\tau,0,\cosh(\tau))$. Now the circle $\R$ generating $\Sigma$   is centered at $(0,0,-\tau \sinh(\tau)+\cosh(\tau))$ and its radius is $r=\tau\cosh(\tau)$. The curvature $\kappa=\widetilde{A}(\widetilde{\nu},\widetilde{\nu})$ is positive because $-\widetilde{N}$ points inwards. The curvature of $\c(\tau)$ is $\kappa=1/\cosh(\tau)^2$.
 We find the solutions of  \eqref{v1}, \eqref{v2} and \eqref{v3}. \begin{enumerate}
\item Solutions of \eqref{v1}.     Since  $\beta<1$, then $\tau<0.765$. In the interval $(0,0.765)$, the equations \eqref{v1} have no solutions.

\item Solutions of \eqref{v2}.    Now $r=1$ gives $\tau_0= 0.765$ and $s_0=0.871$. Then $1/\mu(\tau_0)=1.707$, hence  that $s_0$ does not satisfy   \eqref{v2}. This case is not possible. 

\item Solutions of \eqref{v3}.     From $\beta=\sqrt{r^2-1}$, we deduce that the domain of $\tau$ is $(0.765,\infty)$. Then \eqref{v3} has a unique solution at  $\tau_0= 0.954$. Then  $\beta= 1.013$, in particular $\beta\not=1$. For these values of $r$ and $\beta$, bifurcation occurs according to  Thm.  \ref{t4} applies. The value of the radius of the cylinder that bifurcates is $r=1.423$ and the center of $\c$ is $(0,0,0.435)$.   The function $u_0$ is $\cosh(\beta s)\cos t$.  

\end{enumerate}

\section*{Acknowledgements} This research has been partially supported by MINECO/MICINN/FEDER grant no. PID2023-150727NB-I00, and by the ``Mar\'{\i}a de Maeztu'' Excellence Unit IMAG, reference CEX2020-001105- M, funded by MCINN/AEI/10.13039/501100011033/ CEX2020-001105-M.

%%%%%%%%%%%%%%%%%%%%%%%%%%%%%%%%%%%%%%%%%%%

\end{document}

%% file: fig11.pdf_tex
%% Creator: Inkscape inkscape 0.91, www.inkscape.org
%% Accompanies image file 'fig11.pdf' (pdf, eps, ps)
\begingroup%
  \makeatletter%
  \providecommand\color[2][]{%
    \errmessage{(Inkscape) Color is used for the text in Inkscape, but the package 'color.sty' is not loaded}%
    \renewcommand\color[2][]{}%
  }%
  \providecommand\transparent[1]{%
    \errmessage{(Inkscape) Transparency is used (non-zero) for the text in Inkscape, but the package 'transparent.sty' is not loaded}%
    \renewcommand\transparent[1]{}%
  }%
  \providecommand\rotatebox[2]{#2}%
  \ifx\svgwidth\undefined%
    \setlength{\unitlength}{340.15748031bp}%
    \ifx\svgscale\undefined%
      \relax%
    \else%
      \setlength{\unitlength}{\unitlength * \real{\svgscale}}%
    \fi%
  \else%
    \setlength{\unitlength}{\svgwidth}%
  \fi%
  \global\let\svgwidth\undefined%
  \global\let\svgscale\undefined%
  \makeatother%
  \begin{picture}(1,0.33333333)%
    \put(0,0){\includegraphics[width=\unitlength,page=1]{fig11.pdf}}%
    \put(0.79940788,0.08834236){\color[rgb]{0,0,0}\makebox(0,0)[lb]{\smash{\textbf{$\widetilde{N}$}}}}%
    \put(0.42065875,0.05899159){\color[rgb]{0,0,0}\makebox(0,0)[lb]{\smash{\textbf{$\widetilde{N}$}}}}%
    \put(0,0){\includegraphics[width=\unitlength,page=2]{fig11.pdf}}%
    \put(0.65860011,0.24376283){\color[rgb]{0,0,0}\makebox(0,0)[lb]{\smash{$\Sigma$}}}%
    \put(0,0){\includegraphics[width=\unitlength,page=3]{fig11.pdf}}%
    \put(0.64901624,0.07364021){\color[rgb]{0,0,0}\makebox(0,0)[lb]{\smash{\textbf{$N$}}}}%
    \put(0,0){\includegraphics[width=\unitlength,page=4]{fig11.pdf}}%
    \put(0.47266033,0.11137417){\color[rgb]{0,0,0}\makebox(0,0)[lb]{\smash{\textbf{$N$}}}}%
    \put(0,0){\includegraphics[width=\unitlength,page=5]{fig11.pdf}}%
    \put(0.26828969,0.21072999){\color[rgb]{0,0,0}\makebox(0,0)[lb]{\smash{$\Sigma$}}}%
    \put(0,0){\includegraphics[width=\unitlength,page=6]{fig11.pdf}}%
    \put(0.26889611,0.15162451){\color[rgb]{0,0,0}\makebox(0,0)[lb]{\smash{\textbf{$\Omega$}}}}%
    \put(0.66101979,0.17883056){\color[rgb]{0,0,0}\makebox(0,0)[lb]{\smash{\textbf{$\Omega$}}}}%
    \put(0.39266426,0.26569261){\color[rgb]{0,0,0}\makebox(0,0)[lb]{\smash{\textbf{$S(\mathbf{c})$}}}}%
    \put(0.78228542,0.26330257){\color[rgb]{0,0,0}\makebox(0,0)[lb]{\smash{\textbf{$S(\mathbf{c})$}}}}%
  \end{picture}%
\endgroup%

%% file: fig22.pdf_tex
%% Creator: Inkscape inkscape 0.91, www.inkscape.org
%% Accompanies image file 'fig22.pdf' (pdf, eps, ps)
\begingroup%
  \makeatletter%
  \providecommand\color[2][]{%
    \errmessage{(Inkscape) Color is used for the text in Inkscape, but the package 'color.sty' is not loaded}%
    \renewcommand\color[2][]{}%
  }%
  \providecommand\transparent[1]{%
    \errmessage{(Inkscape) Transparency is used (non-zero) for the text in Inkscape, but the package 'transparent.sty' is not loaded}%
    \renewcommand\transparent[1]{}%
  }%
  \providecommand\rotatebox[2]{#2}%
  \ifx\svgwidth\undefined%
    \setlength{\unitlength}{340.15748031bp}%
    \ifx\svgscale\undefined%
      \relax%
    \else%
      \setlength{\unitlength}{\unitlength * \real{\svgscale}}%
    \fi%
  \else%
    \setlength{\unitlength}{\svgwidth}%
  \fi%
  \global\let\svgwidth\undefined%
  \global\let\svgscale\undefined%
  \makeatother%
  \begin{picture}(1,0.29166667)%
    \put(0.40821428,0.7089153){\color[rgb]{0,0,0}\makebox(0,0)[lb]{\smash{}}}%
    \put(0.96593917,0.77107143){\color[rgb]{0,0,0}\makebox(0,0)[lb]{\smash{}}}%
    \put(0.97937828,0.77107143){\color[rgb]{0,0,0}\makebox(0,0)[lb]{\smash{}}}%
    \put(0.97265876,0.76603168){\color[rgb]{0,0,0}\makebox(0,0)[lb]{\smash{}}}%
    \put(0.96425926,0.762672){\color[rgb]{0,0,0}\makebox(0,0)[lb]{\smash{}}}%
    \put(0.33261905,0.85674603){\color[rgb]{0,0,0}\makebox(0,0)[lt]{\begin{minipage}{0.00335979\unitlength}\raggedright \end{minipage}}}%
    \put(1.38591268,0.11087285){\color[rgb]{0,0,0}\makebox(0,0)[lb]{\smash{}}}%
    \put(0.50127887,0.70321582){\color[rgb]{0,0,0}\makebox(0,0)[lb]{\smash{}}}%
    \put(0,0){\includegraphics[width=\unitlength,page=1]{fig22.pdf}}%
    \put(0.05880918,0.16557083){\color[rgb]{0,0,0}\makebox(0,0)[lb]{\smash{\textbf{$\Sigma$}}}}%
    \put(0.20974644,0.16795071){\color[rgb]{0,0,0}\makebox(0,0)[lb]{\smash{\textbf{$\Sigma$}}}}%
    \put(0,0){\includegraphics[width=\unitlength,page=2]{fig22.pdf}}%
    \put(0.66350852,0.17736678){\color[rgb]{0,0,0}\makebox(0,0)[lb]{\smash{\textbf{$\Sigma$}}}}%
    \put(0.94345289,0.17627813){\color[rgb]{0,0,0}\makebox(0,0)[lb]{\smash{\textbf{$\Sigma$}}}}%
    \put(0.90714718,0.0412901){\color[rgb]{0,0,0}\makebox(0,0)[lb]{\smash{\textbf{$N$}}}}%
    \put(0,0){\includegraphics[width=\unitlength,page=3]{fig22.pdf}}%
    \put(0.83266989,0.2304112){\color[rgb]{0,0,0}\makebox(0,0)[lb]{\smash{\textbf{$\widetilde{N}$}}}}%
    \put(0.1563576,0.24183528){\color[rgb]{0,0,0}\makebox(0,0)[lb]{\smash{\textbf{$S(c)$}}}}%
    \put(0,0){\includegraphics[width=\unitlength,page=4]{fig22.pdf}}%
    \put(0.45649105,0.14719607){\color[rgb]{0,0,0}\makebox(0,0)[lb]{\smash{\textbf{$\Sigma$}}}}%
    \put(0.5507466,0.01111939){\color[rgb]{0,0,0}\makebox(0,0)[lb]{\smash{\textbf{$N$}}}}%
    \put(0,0){\includegraphics[width=\unitlength,page=5]{fig22.pdf}}%
    \put(0.61040597,0.04015236){\color[rgb]{0,0,0}\makebox(0,0)[lb]{\smash{\textbf{$\widetilde{N}$}}}}%
    \put(0,0){\includegraphics[width=\unitlength,page=6]{fig22.pdf}}%
    \put(0.55554282,0.19164713){\color[rgb]{0,0,0}\makebox(0,0)[lb]{\smash{\textbf{$S(c)$}}}}%
    \put(0,0){\includegraphics[width=\unitlength,page=7]{fig22.pdf}}%
    \put(0.69389656,0.21792302){\color[rgb]{0,0,0}\makebox(0,0)[lb]{\smash{\textbf{$S(c)$}}}}%
  \end{picture}%
\endgroup%

%% file: fig2.pdf_tex
%% Creator: Inkscape inkscape 0.91, www.inkscape.org
%% Accompanies image file 'fig2.pdf' (pdf, eps, ps)
\begingroup%
  \makeatletter%
  \providecommand\color[2][]{%
    \errmessage{(Inkscape) Color is used for the text in Inkscape, but the package 'color.sty' is not loaded}%
    \renewcommand\color[2][]{}%
  }%
  \providecommand\transparent[1]{%
    \errmessage{(Inkscape) Transparency is used (non-zero) for the text in Inkscape, but the package 'transparent.sty' is not loaded}%
    \renewcommand\transparent[1]{}%
  }%
  \providecommand\rotatebox[2]{#2}%
  \ifx\svgwidth\undefined%
    \setlength{\unitlength}{340.15748031bp}%
    \ifx\svgscale\undefined%
      \relax%
    \else%
      \setlength{\unitlength}{\unitlength * \real{\svgscale}}%
    \fi%
  \else%
    \setlength{\unitlength}{\svgwidth}%
  \fi%
  \global\let\svgwidth\undefined%
  \global\let\svgscale\undefined%
  \makeatother%
  \begin{picture}(1,0.33333333)%
    \put(0,0){\includegraphics[width=\unitlength,page=1]{fig2.pdf}}%
    \put(0.11594533,0.22900536){\color[rgb]{0,0,0}\makebox(0,0)[lb]{\smash{$\Sigma$}}}%
    \put(0,0){\includegraphics[width=\unitlength,page=2]{fig2.pdf}}%
    \put(0.19481114,0.09671945){\color[rgb]{0,0,0}\makebox(0,0)[lb]{\smash{\textbf{$N$}}}}%
    \put(0.51224821,0.2979341){\color[rgb]{0,0,0}\makebox(0,0)[lb]{\smash{\textbf{$\widetilde{N}$}}}}%
    \put(0.76416515,0.11558178){\color[rgb]{0,0,0}\makebox(0,0)[lb]{\smash{\textbf{$\widetilde{N}$}}}}%
    \put(0.70040699,0.11028753){\color[rgb]{0,0,0}\makebox(0,0)[lb]{\smash{\textbf{$N$}}}}%
    \put(0,0){\includegraphics[width=\unitlength,page=3]{fig2.pdf}}%
    \put(0.44123357,0.13026233){\color[rgb]{0,0,0}\makebox(0,0)[lb]{\smash{\textbf{$N$}}}}%
    \put(0.25007294,0.110713){\color[rgb]{0,0,0}\makebox(0,0)[lb]{\smash{\textbf{$\widetilde{N}$}}}}%
    \put(0,0){\includegraphics[width=\unitlength,page=4]{fig2.pdf}}%
    \put(0.40633228,0.25778123){\color[rgb]{0,0,0}\makebox(0,0)[lb]{\smash{$\Sigma$}}}%
    \put(0.64583992,0.23901379){\color[rgb]{0,0,0}\makebox(0,0)[lb]{\smash{$\Sigma$}}}%
    \put(0.88619853,0.23731191){\color[rgb]{0,0,0}\makebox(0,0)[lb]{\smash{$\Sigma$}}}%
    \put(0,0){\includegraphics[width=\unitlength,page=5]{fig2.pdf}}%
    \put(0.95996935,0.30595084){\color[rgb]{0,0,0}\makebox(0,0)[lb]{\smash{\textbf{$\widetilde{N}$}}}}%
    \put(0,0){\includegraphics[width=\unitlength,page=6]{fig2.pdf}}%
    \put(0.9165857,0.12735726){\color[rgb]{0,0,0}\makebox(0,0)[lb]{\smash{\textbf{$N$}}}}%
    \put(0.11116231,0.01583551){\color[rgb]{0,0,0}\makebox(0,0)[lb]{\smash{\textbf{(a)}}}}%
    \put(0.39440023,0.01498457){\color[rgb]{0,0,0}\makebox(0,0)[lb]{\smash{\textbf{(b)}}}}%
    \put(0.22123676,0.07913468){\color[rgb]{0,0,0}\makebox(0,0)[lb]{\smash{}}}%
    \put(0.63700662,0.01413363){\color[rgb]{0,0,0}\makebox(0,0)[lb]{\smash{\textbf{(c)}}}}%
    \put(0.87871478,0.01243204){\color[rgb]{0,0,0}\makebox(0,0)[lb]{\smash{\textbf{(d)}}}}%
  \end{picture}%
\endgroup%

%% file: fig3.pdf_tex
%% Creator: Inkscape inkscape 0.91, www.inkscape.org
%% Accompanies image file 'fig3.pdf' (pdf, eps, ps)
\begingroup%
  \makeatletter%
  \providecommand\color[2][]{%
    \errmessage{(Inkscape) Color is used for the text in Inkscape, but the package 'color.sty' is not loaded}%
    \renewcommand\color[2][]{}%
  }%
  \providecommand\transparent[1]{%
    \errmessage{(Inkscape) Transparency is used (non-zero) for the text in Inkscape, but the package 'transparent.sty' is not loaded}%
    \renewcommand\transparent[1]{}%
  }%
  \providecommand\rotatebox[2]{#2}%
  \ifx\svgwidth\undefined%
    \setlength{\unitlength}{340.15748031bp}%
    \ifx\svgscale\undefined%
      \relax%
    \else%
      \setlength{\unitlength}{\unitlength * \real{\svgscale}}%
    \fi%
  \else%
    \setlength{\unitlength}{\svgwidth}%
  \fi%
  \global\let\svgwidth\undefined%
  \global\let\svgscale\undefined%
  \makeatother%
  \begin{picture}(1,0.33333333)%
    \put(0,0){\includegraphics[width=\unitlength,page=1]{fig3.pdf}}%
    \put(0.11830086,0.18430913){\color[rgb]{0,0,0}\makebox(0,0)[lb]{\smash{$\Sigma$}}}%
    \put(0,0){\includegraphics[width=\unitlength,page=2]{fig3.pdf}}%
    \put(0.19481114,0.08731204){\color[rgb]{0,0,0}\makebox(0,0)[lb]{\smash{\textbf{$N$}}}}%
    \put(0.49397812,0.30845298){\color[rgb]{0,0,0}\makebox(0,0)[lb]{\smash{\textbf{$\widetilde{N}$}}}}%
    \put(0.70040699,0.10088012){\color[rgb]{0,0,0}\makebox(0,0)[lb]{\smash{\textbf{$N$}}}}%
    \put(0.52369884,0.14029287){\color[rgb]{0,0,0}\makebox(0,0)[lb]{\smash{\textbf{$N$}}}}%
    \put(0,0){\includegraphics[width=\unitlength,page=3]{fig3.pdf}}%
    \put(0.39692487,0.23426272){\color[rgb]{0,0,0}\makebox(0,0)[lb]{\smash{$\Sigma$}}}%
    \put(0.63643251,0.20608787){\color[rgb]{0,0,0}\makebox(0,0)[lb]{\smash{$\Sigma$}}}%
    \put(0.88149483,0.22790451){\color[rgb]{0,0,0}\makebox(0,0)[lb]{\smash{$\Sigma$}}}%
    \put(0.96882468,0.27744885){\color[rgb]{0,0,0}\makebox(0,0)[lb]{\smash{\textbf{$\widetilde{N}$}}}}%
    \put(0,0){\includegraphics[width=\unitlength,page=4]{fig3.pdf}}%
    \put(0.90442129,0.12741094){\color[rgb]{0,0,0}\makebox(0,0)[lb]{\smash{\textbf{$N$}}}}%
    \put(0.08510829,0.04082422){\color[rgb]{0,0,0}\makebox(0,0)[lb]{\smash{\textbf{(a)}}}}%
    \put(0.39519127,0.04023109){\color[rgb]{0,0,0}\makebox(0,0)[lb]{\smash{\textbf{(b)}}}}%
    \put(0.41723823,0.04428396){\color[rgb]{0,0,0}\makebox(0,0)[lb]{\smash{}}}%
    \put(0.64020927,0.04050698){\color[rgb]{0,0,0}\makebox(0,0)[lb]{\smash{\textbf{(c)}}}}%
    \put(0.86941927,0.04092327){\color[rgb]{0,0,0}\makebox(0,0)[lb]{\smash{\textbf{(d)}}}}%
    \put(0.25802917,0.10482935){\color[rgb]{0,0,0}\makebox(0,0)[lb]{\smash{\textbf{$\widetilde{N}$}}}}%
    \put(0.75969428,0.12160251){\color[rgb]{0,0,0}\makebox(0,0)[lb]{\smash{\textbf{$\widetilde{N}$}}}}%
  \end{picture}%
\endgroup%

%% file: fig4.pdf_tex
%% Creator: Inkscape inkscape 0.91, www.inkscape.org
%% Accompanies image file 'fig4.pdf' (pdf, eps, ps)
\begingroup%
  \makeatletter%
  \providecommand\color[2][]{%
    \errmessage{(Inkscape) Color is used for the text in Inkscape, but the package 'color.sty' is not loaded}%
    \renewcommand\color[2][]{}%
  }%
  \providecommand\transparent[1]{%
    \errmessage{(Inkscape) Transparency is used (non-zero) for the text in Inkscape, but the package 'transparent.sty' is not loaded}%
    \renewcommand\transparent[1]{}%
  }%
  \providecommand\rotatebox[2]{#2}%
  \ifx\svgwidth\undefined%
    \setlength{\unitlength}{340.15748031bp}%
    \ifx\svgscale\undefined%
      \relax%
    \else%
      \setlength{\unitlength}{\unitlength * \real{\svgscale}}%
    \fi%
  \else%
    \setlength{\unitlength}{\svgwidth}%
  \fi%
  \global\let\svgwidth\undefined%
  \global\let\svgscale\undefined%
  \makeatother%
  \begin{picture}(1,0.29166667)%
    \put(0,0){\includegraphics[width=\unitlength,page=1]{fig4.pdf}}%
    \put(0.09027887,0.18878031){\color[rgb]{0,0,0}\makebox(0,0)[lb]{\smash{$\Sigma$}}}%
    \put(0.47650998,0.24999878){\color[rgb]{0,0,0}\makebox(0,0)[lb]{\smash{\textbf{$\widetilde{N}$}}}}%
    \put(0.71942119,0.06263543){\color[rgb]{0,0,0}\makebox(0,0)[lb]{\smash{\textbf{$\widetilde{N}$}}}}%
    \put(0.23484318,0.0839353){\color[rgb]{0,0,0}\makebox(0,0)[lb]{\smash{\textbf{$\widetilde{N}$}}}}%
    \put(0,0){\includegraphics[width=\unitlength,page=2]{fig4.pdf}}%
    \put(0.3850525,0.21299846){\color[rgb]{0,0,0}\makebox(0,0)[lb]{\smash{$\Sigma$}}}%
    \put(0.63643251,0.18256938){\color[rgb]{0,0,0}\makebox(0,0)[lb]{\smash{$\Sigma$}}}%
    \put(0.87679112,0.1855712){\color[rgb]{0,0,0}\makebox(0,0)[lb]{\smash{$\Sigma$}}}%
    \put(0.9633256,0.20265902){\color[rgb]{0,0,0}\makebox(0,0)[lb]{\smash{\textbf{$\widetilde{N}$}}}}%
    \put(0,0){\includegraphics[width=\unitlength,page=3]{fig4.pdf}}%
    \put(0.38762616,0.08107415){\color[rgb]{0,0,0}\makebox(0,0)[lb]{\smash{\textbf{$N$}}}}%
    \put(0,0){\includegraphics[width=\unitlength,page=4]{fig4.pdf}}%
    \put(0.6399923,0.05129016){\color[rgb]{0,0,0}\makebox(0,0)[lb]{\smash{\textbf{$N$}}}}%
    \put(0,0){\includegraphics[width=\unitlength,page=5]{fig4.pdf}}%
    \put(0.87318928,0.06585459){\color[rgb]{0,0,0}\makebox(0,0)[lb]{\smash{\textbf{$N$}}}}%
    \put(0.07962165,0.00007222){\color[rgb]{0,0,0}\makebox(0,0)[lb]{\smash{\textbf{(a)}}}}%
    \put(0.38970463,-0.00052091){\color[rgb]{0,0,0}\makebox(0,0)[lb]{\smash{\textbf{(b)}}}}%
    \put(0.63472267,-0.00024501){\color[rgb]{0,0,0}\makebox(0,0)[lb]{\smash{\textbf{(c)}}}}%
    \put(0.88274755,0.00017127){\color[rgb]{0,0,0}\makebox(0,0)[lb]{\smash{\textbf{(d)}}}}%
    \put(0,0){\includegraphics[width=\unitlength,page=6]{fig4.pdf}}%
    \put(0.17809866,0.03610102){\color[rgb]{0,0,0}\makebox(0,0)[lb]{\smash{\textbf{$N$}}}}%
    \put(0,0){\includegraphics[width=\unitlength,page=7]{fig4.pdf}}%
  \end{picture}%
\endgroup%

%% file: figure-dib.pdf_tex
%% Creator: Inkscape inkscape 0.91, www.inkscape.org
%% Accompanies image file 'figure-dib.pdf' (pdf, eps, ps)
\begingroup%
  \makeatletter%
  \providecommand\color[2][]{%
    \errmessage{(Inkscape) Color is used for the text in Inkscape, but the package 'color.sty' is not loaded}%
    \renewcommand\color[2][]{}%
  }%
  \providecommand\transparent[1]{%
    \errmessage{(Inkscape) Transparency is used (non-zero) for the text in Inkscape, but the package 'transparent.sty' is not loaded}%
    \renewcommand\transparent[1]{}%
  }%
  \providecommand\rotatebox[2]{#2}%
  \ifx\svgwidth\undefined%
    \setlength{\unitlength}{359.99999459bp}%
    \ifx\svgscale\undefined%
      \relax%
    \else%
      \setlength{\unitlength}{\unitlength * \real{\svgscale}}%
    \fi%
  \else%
    \setlength{\unitlength}{\svgwidth}%
  \fi%
  \global\let\svgwidth\undefined%
  \global\let\svgscale\undefined%
  \makeatother%
  \begin{picture}(1,0.39370079)%
    \put(0,0){\includegraphics[width=\unitlength,page=1]{figure-dib.pdf}}%
    \put(0.14208564,0.23752632){\color[rgb]{0,0,0}\makebox(0,0)[lb]{\smash{\textbf{$\beta_{2,2}$}}}}%
    \put(0.28208598,0.2610707){\color[rgb]{0,0,0}\makebox(0,0)[lb]{\smash{\textbf{$\beta_{2,3}$}}}}%
    \put(0.41754243,0.28901938){\color[rgb]{0,0,0}\makebox(0,0)[lb]{\smash{\textbf{$\beta_{2,4}$}}}}%
    \put(0.12678054,0.13353666){\color[rgb]{0,0,0}\makebox(0,0)[lb]{\smash{\textbf{$\beta_{1,2}$}}}}%
    \put(0.26171204,0.13497451){\color[rgb]{0,0,0}\makebox(0,0)[lb]{\smash{\textbf{$\beta_{1,3}$}}}}%
    \put(0.40734216,0.14534126){\color[rgb]{0,0,0}\makebox(0,0)[lb]{\smash{\textbf{$\beta_{1,4}$}}}}%
    \put(0,0){\includegraphics[width=\unitlength,page=2]{figure-dib.pdf}}%
    \put(0.78205508,0.2433868){\color[rgb]{0,0,0}\makebox(0,0)[lb]{\smash{\textbf{$\beta_{1,3}$}}}}%
    \put(0.75577491,0.13436693){\color[rgb]{0,0,0}\makebox(0,0)[lb]{\smash{\textbf{$\beta_{2,3}$}}}}%
    \put(0.64962785,0.25607662){\color[rgb]{0,0,0}\makebox(0,0)[lb]{\smash{\textbf{$\beta_{1,2}$}}}}%
    \put(0.8794573,0.12257884){\color[rgb]{0,0,0}\makebox(0,0)[lb]{\smash{\textbf{$\beta_{2,4}$}}}}%
    \put(0.63446389,0.15102969){\color[rgb]{0,0,0}\makebox(0,0)[lb]{\smash{\textbf{$\beta_{2,2}$}}}}%
    \put(0.90719092,0.23710999){\color[rgb]{0,0,0}\makebox(0,0)[lb]{\smash{\textbf{$\beta_{1,4}$}}}}%
    \put(0,0){\includegraphics[width=\unitlength,page=3]{figure-dib.pdf}}%
    \put(0.52371398,0.27512229){\color[rgb]{0,0,0}\makebox(0,0)[lb]{\smash{\textbf{$\beta_{1,1}$}}}}%
  \end{picture}%
\endgroup%

%% file: fig5.pdf_tex
%% Creator: Inkscape inkscape 0.91, www.inkscape.org
%% Accompanies image file 'fig5.pdf' (pdf, eps, ps)
\begingroup%
  \makeatletter%
  \providecommand\color[2][]{%
    \errmessage{(Inkscape) Color is used for the text in Inkscape, but the package 'color.sty' is not loaded}%
    \renewcommand\color[2][]{}%
  }%
  \providecommand\transparent[1]{%
    \errmessage{(Inkscape) Transparency is used (non-zero) for the text in Inkscape, but the package 'transparent.sty' is not loaded}%
    \renewcommand\transparent[1]{}%
  }%
  \providecommand\rotatebox[2]{#2}%
  \ifx\svgwidth\undefined%
    \setlength{\unitlength}{340.15748031bp}%
    \ifx\svgscale\undefined%
      \relax%
    \else%
      \setlength{\unitlength}{\unitlength * \real{\svgscale}}%
    \fi%
  \else%
    \setlength{\unitlength}{\svgwidth}%
  \fi%
  \global\let\svgwidth\undefined%
  \global\let\svgscale\undefined%
  \makeatother%
  \begin{picture}(1,0.33333333)%
    \put(0,0){\includegraphics[width=\unitlength,page=1]{fig5.pdf}}%
    \put(0.80767778,0.10317484){\color[rgb]{0,0,0}\makebox(0,0)[lb]{\smash{$\Sigma$}}}%
    \put(0,0){\includegraphics[width=\unitlength,page=2]{fig5.pdf}}%
    \put(0.68414284,0.29271407){\color[rgb]{0,0,0}\makebox(0,0)[lb]{\smash{\textbf{$\widetilde{N}$}}}}%
    \put(0.35911344,0.10605752){\color[rgb]{0,0,0}\makebox(0,0)[lb]{\smash{\textbf{$\widetilde{N}$}}}}%
    \put(0,0){\includegraphics[width=\unitlength,page=3]{fig5.pdf}}%
    \put(0.25550603,0.24418474){\color[rgb]{0,0,0}\makebox(0,0)[lb]{\smash{$\Sigma$}}}%
    \put(0,0){\includegraphics[width=\unitlength,page=4]{fig5.pdf}}%
    \put(0.79010485,0.29002345){\color[rgb]{0,0,0}\makebox(0,0)[lb]{\smash{\textbf{$S(\mathbf{c})$}}}}%
    \put(0.35221758,0.23468948){\color[rgb]{0,0,0}\makebox(0,0)[lb]{\smash{\textbf{$S(\mathbf{c})$}}}}%
    \put(0,0){\includegraphics[width=\unitlength,page=5]{fig5.pdf}}%
    \put(0.26208046,0.09785475){\color[rgb]{0,0,0}\makebox(0,0)[lb]{\smash{\textbf{$N$}}}}%
    \put(0,0){\includegraphics[width=\unitlength,page=6]{fig5.pdf}}%
    \put(0.70455395,0.15316389){\color[rgb]{0,0,0}\makebox(0,0)[lb]{\smash{\textbf{$N$}}}}%
  \end{picture}%
\endgroup%

%% file: channel-stability.bbl
\begin{thebibliography}{99}
 



   \bibitem{ad}  Adamson, A. W.,  Gast, A. P.:    Physical chemistry of surfaces.
(sixth ed.), John Wiley $\&$ Sons. (1997)

\bibitem{alp}  Al\'{\i}as, L., L\'opez, R.,    Palmer, B.:  Stable constant mean curvature surfaces with circular boundary.  Proc.A.M.S.  127,   1195--1200 (1999)

 \bibitem{and} Andersson, S., Hyde, S. T., Larsson, K., et al.:  Minimal surfaces and structures: From inorganic and metal crystals to cell membranes and biopolymers. Chem Rev.  88, 221--242 (1988)

 \bibitem{bc}  Barbosa, J. L.,  do Carmo M.: 
 Stability of hypersurfaces with constant mean curvature.
 Math. Z. 185,  339--353 (1984)
 
 \bibitem{benilov}  Benilov, E.:  On the stability of shallow rivulets.  J. Fluid Mech. 636,  455--474 (2009)



 \bibitem{bst1}    Bostwick, J. B., Steen, P. H.:   Stability of constrained capillary surfaces.  Annu. Rev. Fluid Mech.  47,   539--568 (2015)
 
  \bibitem{bkk}   Brinkmann, M., Lipowsky, R.: 
Wetting morphologies on substrates with striped surface domains. 
J. Appl. Phys. 92,  4296--4306 (2002)

 






  \bibitem{bs}   Brown, R. A., Scriven,  L. E.: 
 On the multiple equilibrium shapes and stability of an
interface pinned on a slot.  J. Colloid Interface Sci. 78,   528--542 (1980)


\bibitem{bru}  Brubaker, N. D.:  Shapes of large, static soap bubbles. Proc. R. Soc. A 477: 20200851 (2020)



\bibitem{concus} Concus, P., Finn,  R.:  On the behavior of a capillary surface in a wedge.   Proc. Natl Acad. Sci. USA  63,  292--299 (1969)



\bibitem{cr}   Crandall, M., Rabinowitz, P.: 
 Bifurcation from simple eigenvalues. 
 J. Functional Analysis 8,  321--340 (1971)
 
 \bibitem{dav}    Davis, S. H.:  Moving contact lines and rivulet instabilities. Part 1: The
static rivulet.    J. Fluid Mech.  98,  225--242 (1980)

   \bibitem{dege}    de Gennes, P-G., Brochard-Wyart, F., Quer\'e D.:    Capillarity and wetting phenomena. Springer-Verlag, New York (2004) 
   
   \bibitem{ghll}  Gau, H., Herminghaus, S., Lenz, P., Lipowsky,  R.: 
 Liquid microchannels on structured surfaces. 
Science  283,  46--49 (1999)

   \bibitem{herrada}  Herrada, M. A., Mohamed, A. S. , Montanero, J. M., Gan\'an-Calvo,  A. M.:  Stability of a rivulet flowing in a microchannel. Int. J.  Multiphase Flow  69,  1--7 (2015)
   
\bibitem{hyd} Hyde, S., Blum, Z., Landh, T., et al.:    The language of shape: the role of curvature in condensed matter: physics, chemistry and biology. Amsterdam: Elsevier (1996)

\bibitem{hys} Hyde, S. T., Schr\"{o}der-Turk, G. E.: Geometry of interfaces: topological complexity in biology and materials.  Interface Focus.   2, 529--538 (2012) 

\bibitem{ko}   Koiso, M.: 
 Deformation and stability of surfaces with constant mean curvature. Tohoku Math. J. 54. 145--159 (2022)

\bibitem{kpm1}  Koiso, M., Palmer, B., Piccione, P.:  Bifurcation and symmetry breaking of nodoids with fixed boundary.    Adv. Calc. Var. 8,  337--370 (2015)

\bibitem{kpm2}  Koiso, M., Palmer, B., Piccione, P.:  Stability and bifurcation for surfaces with constant mean curvature.     J. Math. Soc. Japan  69 1519--1554 (2017)

\bibitem{kmlr}  
  Kusumaatmaja, H.,  Lipowsky, R.,   Jin, C., Mutihac, R. C.,   Riegler, H.:  Nonisomorphic nucleation pathways arising from morphological transitions of liquid channels. 
Physical Review Letters 108, 126102 (2012)

\bibitem{lan}  Langbein, D.: 
 The shape and stability of liquid menisci in solid edges.  J. Fluid Mech. 213,  251--265 (1990)
 
 \bibitem{la}      Laplace,  P. S.:   On capillary attraction. Supplement to the tenth book of the M\'echanique c\'eleste, translated by N. Bowditch, Vol. IV, pp.685--1018,Chelsea,  New York (1966)  
 
\bibitem{li}    Lipowsky, R.: 
 Structured surfaces and morphological wetting transitions,
Interface Science 9,  105--115 (2001)



\bibitem{lo1}    L\'opez, R.: 
 Bifurcation of cylinders for wetting and dewetting models with striped geometry. SIAM J. Math. Analysis 44,  946--965 (2012)

\bibitem{lo0}  L\'opez, R.:   Constant mean curvature surfaces with boundary. Springer Science, New York (2013)

 \bibitem{lo00}    L\'opez, R.: 
 Capillary surfaces with free boundary in a wedge.  Adv. Math. 262, 476--483 (2014)

  
\bibitem{lo2}   L\'opez, R.: 
 Stability and bifurcation of a capillary surface on a cylinder.  SIAM J. Appl. Math. 77, 108--127 (2017)
 
\bibitem{maj}  Majumbar, S. R.,  Michael, D. H.:    The equilibrium and stability of two-dimensional pendent drop.    Proc. R. Soc. Lond. A 351,  89--115 (1976)


    \bibitem{pl} Plateau, J. A. F.:  Statique exp\'erimentale et th\'eorique des liquides soumis aux seules forces mol\'eculaires. vol. 2. Gauthier-Villars (2018)

 
\bibitem{ra}  Rayleigh, J. W. S.:  On the instability of jets. Proc. London Math. Soc. 10, 4--13 (1879)

 \bibitem{rv}  Ros, A.,  Vergasta, E.: Stability for hypersurfaces of constant mean curvature with free boundary.    Geom. Dedicata  56,   19--33 (1995)
 
 \bibitem{rss}   Roy, R.,  Schwartz, L. W.: 
 On the stability of liquid ridges.   J. Fluid Mech. 391,    293--318 (1999)
 
 \bibitem{smo} Smoller, J.: Shock waves and reaction-diffusion equations. Springer-Verlag, New York (1983)
 
 \bibitem{sl}     Speth, R. L., Lauga, E.: 
 Capillary instability on a hydrophilic stripe. New J. Phys. 11, 075024 (2009)


  \bibitem{tho}  Thomas, E. L., Anderson, D. M., Henkee, C. S. et al.:  Periodic area-minimizing surfaces in block copolymers.  Nature 334, 598--601 (1988)
  
  \bibitem{ubal}  Ubal, S., Grassia, P., Campana, D., Giavedoni, M., Saita, F.:  The influence of inertia and contact angle on the
instability of partially wetting liquid strips: A numerical analysis study.  Phys. Fluids    26, 032106 (2014)


 \bibitem{vo5}   Vogel, T. I.: Stability of a liquid drop trapped between two parallel planes.  SIAM J. Appl. Math. 47, 516--525 (1987)
 
    \bibitem{vo6}   Vogel, T. I.:  Stability of a liquid drop trapped between two parallel planes II: general contact angles.   SIAM J. Appl. Math.  49, 1009--1028 (1989)
    
    \bibitem{vo}   Vogel, T. I.: Stability of a surface of constant mean curvature in a wedge.
Indiana Univ. Math. J. 41,   625--648 (1992)


\bibitem{vo2}   Vogel, T. I.:  Sufficient conditions for capillary surfaces to be energy minima.
Pacific J. Math. 194,   469--489 (2000)

 
\bibitem{vo12} T. I. Vogel, Comments on radially symmetric liquid bridges with inflected profiles. Discrete Contin. Dyn. Syst. 2005, suppl., 862--867.


     \bibitem{vo3}   Vogel, T. I.:   Liquid bridges between balls: the small volume instability.    J. Math. Fluid Mech. 15,  397--413 (2013)



 
  \bibitem{vo4}   Vogel, T. I.:     Capillary surfaces in circular cylinders.  J. Math. Fluid Mech. 23,  Paper No. 68, 13 pp. (2021)
  
 

\bibitem{we}   Wente,  H. C.:    The symmetry of sessile and pendent drops. Pacific J. Math. 88, 387--397 (1980)

\bibitem{we-c} Wente, H. C.:  The capillary problem for an infinite trough.  Calc. Var. 3, 155--192 (1995)

\bibitem{we-c2} Wente, H. C.: A surprising bubble catastrophe.  Pac. J. Math. 189, 339--375 (1999)

  



\bibitem{yh}   Yang, L.,    Homsy, G. M.: 
  Capillary instabilities of liquid films inside a wedge.   Phys. Fluids  19,  044101 (2007)
  
\bibitem{yd} Young,  G. W., Davis, S. H.:   Rivulet instabilities. J. Fluid Mech. 176, 1--31 (1987)

\bibitem{yo} Young,  T.:   An essay on the cohesion of fluids.   Philos. Trans. Roy. Soc. London  Ser. A.   95, 65--87 (1805)


 
 
  \bibitem{zz} Zhou, X., Zhang, F.:   Bifurcation of a partially immersed plate between two parallel plates.   J. Fluid Mech. 817, 122--137 (2017)


 
 





\end{thebibliography}
